\documentclass{fancypaper}
\usepackage[extdef=true]{delimset}
\DeclareMathDelimiterSet{\scal}[2]{
  \selectdelim[l]< {#1} \mathpunct{}\selectdelim[p]| {#2} \selectdelim[r]>
}

\newcommand{\menge}[2]{
  \bigl\{{#1}\mid{#2}\bigr\}
}
\DeclareMathDelimiterSet{\Menge}[2]{
  \selectdelim[l]\{{#1}\selectdelim[m]|{#2}\selectdelim[r]\}
}
\newcommand*\Cdot{{\mkern 1.6mu\cdot\mkern 1.6mu}}
\DeclareMathDelimiterSet{\killing}[2]{
  \selectdelim[l]( {#1} \mathpunct{}\selectdelim[p]| {#2} \selectdelim[r])
}

\newcommand{\AD}{\mathscr{A}}
\newcommand{\DD}{\mathscr{D}}

\newcommand{\SX}{\mathscr{S}}
\newcommand{\ad}{\mathscr{a}}
\newcommand{\bd}{\mathscr{b}}

\newcommand{\CC}{\mathbb{C}}
\newcommand{\KB}{\mathbb{K}}
\newcommand{\HB}{\mathbb{H}}
\newcommand{\NN}{\mathbb{N}}
\newcommand{\RR}{\mathbb{R}}
\newcommand{\SB}{\mathbb{S}}

\newcommand{\HH}{\mathcal{H}}
\newcommand{\GG}{\mathcal{G}}
\newcommand{\UU}{\mathcal{U}}

\newcommand{\XX}{\mathcal{X}}

\newcommand{\HS}{\mathsf{H}}
\newcommand{\US}{\mathsf{U}}
\newcommand{\PS}{\mathsf{P}}
\newcommand{\SD}{\mathsf{S}}
\newcommand{\SO}{\mathsf{SO}}

\newcommand{\ts}{\mathsf{t}}
\newcommand{\sd}{\mathsf{s}}

\newcommand{\FH}{\mathfrak{H}}

\newcommand{\BE}{\EuScript{B}}
\newcommand{\FF}{\EuScript{F}}

\newcommand{\ii}{\mathrm{i}}    
\newcommand{\ee}{\mathrm{e}}    
\newcommand{\pinf}{{+}\infty}
\newcommand{\minf}{{-}\infty}
\newcommand{\emp}{\varnothing}
\newcommand{\exi}{\exists\,}
\newcommand{\weakly}{\rightharpoonup}

\renewcommand{\leq}{\leqslant}
\renewcommand{\geq}{\geqslant}
\newcommand\restr[2]{{
  \left.\kern-\nulldelimiterspace 
    #1 
  \right|_{#2} 
  }
}

\newcommand{\zeroun}{\intv[o]{0}{1}}
\newcommand{\rzeroun}{\intv[l]{0}{1}}

\newcommand{\RXX}{\overline{\RR}}
\newcommand{\RX}{\intv[l]{\minf}{\pinf}}
\newcommand{\RP}{\RR_+}
\newcommand{\RPP}{\intv[o]{0}{\pinf}}
\newcommand{\RPX}{\intv{0}{\pinf}}

\DeclareMathOperator{\aff}{aff}

\DeclareMathOperator{\ran}{range}

\DeclareMathOperator{\intdom}{int\,dom}

\DeclareMathOperator{\sign}{sign}
\DeclareMathOperator{\Rank}{Rank}

\newcommand{\Sum}{\displaystyle\sum}
\newcommand{\Id}{\mathrm{Id}}

\DeclareMathOperator{\conv}{conv}

\DeclareMathOperator{\dom}{dom}

\DeclareMathOperator{\env}{env}

\DeclareMathOperator{\gra}{gra}

\DeclareMathOperator{\Prox}{Prox}
\DeclareMathOperator{\Proj}{Proj}
\DeclareMathOperator{\dist}{dist}

\newcommand{\minimize}[2]{
  \underset{\substack{{#1}}}{\operatorname{minimize}}\;\;#2
}

\DeclareMathOperator{\Argmin}{Argmin}

\newcommand*{\tran}{^{\mkern-1.6mu\mathsf{T}}}
\renewcommand{\Re}{\operatorname{Re}}
\DeclareMathOperator{\tra}{trace}
\DeclareMathOperator{\Tra}{Trace}

\DeclareMathOperator{\Diag}{Diag}

\makeatletter
\def\upintkern@{\mkern-7mu\mathchoice{\mkern-3.5mu}{}{}{}}
\def\upintdots@{\mathchoice{\mkern-4mu\@cdots\mkern-4mu}%
{{\cdotp}\mkern1.5mu{\cdotp}\mkern1.5mu{\cdotp}}%
{{\cdotp}\mkern1mu{\cdotp}\mkern1mu{\cdotp}}%
{{\cdotp}\mkern1mu{\cdotp}\mkern1mu{\cdotp}}}
\makeatother
\DeclareFontFamily{OMX}{mdbch}{}
\DeclareFontShape{OMX}{mdbch}{m}{n}{ <->s * [0.9]  mdbchr7v }{}
\DeclareFontShape{OMX}{mdbch}{b}{n}{ <->s * [0.9]  mdbchb7v }{}
\DeclareFontShape{OMX}{mdbch}{bx}{n}{<->ssub * mdbch/b/n}{}
\DeclareSymbolFont{uplargesymbols}{OMX}{mdbch}{m}{n}
\SetSymbolFont{uplargesymbols}{bold}{OMX}{mdbch}{b}{n}
\DeclareMathSymbol{\upintop}{\mathop}{uplargesymbols}{82}
\DeclareMathSymbol{\upointop}{\mathop}{uplargesymbols}{"48}
\makeatletter
\renewcommand{\int}{\DOTSI\upintop\ilimits@}
\renewcommand{\oint}{\DOTSI\upointop\ilimits@}
\makeatother

\begin{document}

\author{H\`oa T. B\`ui}
\affil{%
  Curtin University
  \affilcr
  Centre for Optimisation and Decision Science
  \affilcr
  Kent Street, Bentley, Western Australia 6102, Australia
  \affilcr
  \email{hoa.bui@curtin.edu.au}
}

\author{Minh N. B\`ui}
\affil{%
  University of Graz
  \affilcr
  Department of Mathematics and Scientific Computing, NAWI Graz
  \affilcr
  Heinrichstra\ss{}e 36, 8010 Graz, Austria
  \affilcr
  \email{minh.bui@uni-graz.at}
}

\author{Christian Clason}
\affil{%
  University of Graz
  \affilcr
  Department of Mathematics and Scientific Computing, NAWI Graz
  \affilcr
  Heinrichstra\ss{}e 36, 8010 Graz, Austria
  \affilcr
  \email{c.clason@uni-graz.at}
}

\title{Convex Analysis in\\ Spectral Decomposition Systems%
    \thanks{%
        Corresponding author: M.~N.~B\`ui.
        The work of H.~T.~B\`ui was supported by the Australian
        Research Council (grant number DE260100027).
        This research was funded in whole or in part by the
        Austrian Science Fund (FWF)
        \href{https://dx.doi.org/10.55776/F100800}{10.55776/F100800}.%
    }
}

\date{~}

\thispagestyle{plain.scrheadings}

\maketitle

\begin{abstract}
This work is concerned with the convex analysis of functions
defined on (not necessarily finite-dimensional) Hilbert spaces
whose values depend solely on a
certain ``spectrum'' of the arguments, a class we term ``spectral
functions.'' We propose a notion
of a spectral decomposition system which brings together a wide
array of settings underlying important applications such as
Fourier-phase-invariant functions, mixed-norm regularization, and
functions of eigenvalues or (signed) singular values of matrices.
We are particularly motivated by algorithmic requirements
for evaluating convex analytical objects.
Thus, a central contribution is a novel reduced minimization
principle that enables the constructive reduction of minimization
problems involving spectral functions to those of the simpler
associated invariant functions. This result is then leveraged to
explicitly evaluate the conjugates, subgradients, and
set-valued Bregman proximity operators of spectral functions.
\end{abstract}


\newpage

\section{Introduction}

A wide array of practically relevant optimization problems
are defined on (not necessarily finite-dimensional) Hilbert
spaces of matrices, operators, or functions instead of vectors.
In these settings, a frequently encountered class of functions is
those whose values depend solely on a certain \emph{spectrum} of
their arguments, which we shall term \emph{spectral functions}.
Instances of such spectral functions are ubiquitous in
application areas as diverse as
phase retrieval \cite{Bauschke02,Luke02},
robust matrix estimation \cite{Benfenati20},
matrix completion \cite{Candes10},
signal processing \cite{Chan16},
image recovery \cite{Chen19},
conic programming in Euclidean Jordan algebras \cite{Coey23},
and nonlinear elasticity \cite{Rosakis98}.

Structurally, such a spectral function $\Phi$ can be represented
as the composition of a ``simpler'' \emph{associated invariant function}
$\varphi$ defined on another Hilbert space and a
\emph{spectral mapping} $\gamma$.
Thus, of particular importance in these applications are
constructive formulae to evaluate (convex) subgradients and
(Bregman) proximity operators of $\Phi$ at a given point $X$
through a systematic reduction:
\begin{enumerate}[label={\normalfont\arabic*\autodot}]
\item
Compute the corresponding object of the simpler associated invariant
function $\varphi$ at $\gamma(X)$.
\item
Lift the result back to the original space through a suitable
embedding operator to obtain the desired result.
\end{enumerate}
This reduction-lifting mechanism has received
significant attention in the literature, as these formulae --
especially for proximity operators -- serve as
the primary algorithmic building blocks for modern first-order
nonsmooth optimization methods.
The two most notable settings are convex spectral functions of
eigenvalues of Hermitian matrices or singular values of
rectangular matrices
\cite{Livre1,Beck17,Benfenati20,Chan16,Drus18,Lewis95,Lewis96a,Mordukhovich22}.
In the setting of Euclidean Jordan algebras,
a general formula for subgradients of spectral functions of
eigenvalues is established in \cite{Baes07}, while the
corresponding one for
proximity operators remains unaddressed.

However, each of these works treats a specific setting in
isolation, and requires setting-specific techniques to derive the
desired formulae for convex analytical objects. To provide a
unified perspective, the framework of a
\emph{normal decomposition system}
\cite{Lewis96b} was developed to study the convex analysis of
spectral functions across several settings,
which captures in particular the cases of
Hermitian and rectangular matrices.
Nevertheless, this framework does not encompass
Euclidean Jordan algebras in general \cite{Orlitzky22}, nor
does it provide a general formula for proximity operators.
While the more recent Fan--Theobald--von~Neumann framework
\cite{Gowda19} offers a wider reach,
its abstract nature allows only for nonconstructive
characterization of subdifferentials, leaving constructive formulae
and proximity operators unaddressed.

It is therefore the goal of this work to:
\begin{enumerate}[label={\normalfont\bfseries G\arabic*\autodot}]
\item
Develop a general framework of \emph{spectral decomposition system}
(see \cref{d:2} for a precise definition)
that encompasses a broad spectrum of existing frameworks, such as
Hermitian matrices, rectangular matrices, normal decomposition
systems, Euclidean Jordan algebras, as well as infinite-dimensional
constructs such as Fourier-phase-invariant and block-radial
functions.
\item
\label{G2}
Central to our framework is the explicit incorporation of a family
of embedding operators and the spectral decomposition property,
which are motivated by the various decomposition results
for matrices.
Together, these structural ingredients enable a \emph{reduced
minimization principle}
(\cref{t:2}), which provides a systematic lifting mechanism to
evaluate the conjugate (\cref{p:8}),
subgradients (\cref{p:9,p:32}), and Bregman proximity operators
(\cref{t:4}) of spectral functions.
Particularly,
we develop for the first time a formula that fully
describes -- and not merely a selection of -- the
set-valued Bregman proximity operators of nonconvex spectral
functions.
Our approach yields constructive formulae that are readily
implementable for optimization algorithms, analogous to
the classical results for matrices.
\end{enumerate}
We will discuss how our framework relates to existing works, in
particular to the normal decomposition system and
Fan--Theobald--von~Neumann frameworks, in connection to several
individual results throughout this paper.

This work is structured as follows. In the next \cref{sec:2}, we
present in \cref{d:2} the precise definition of a spectral
decomposition system and illustrate its breadth through a diverse
range of examples, including those settings previously studied in
isolation.
\Cref{sec:3} collects essential technical results on spectral
mappings and spectral-induced ordering mappings that underpin
the subsequent analysis.
\Cref{sec:4} is devoted to the
reduced minimization principle (\cref{t:2}) and its applications to
evaluating the conjugate and convex subdifferential of a
spectral function (\cref{p:8} and \cref{p:9}, respectively).
Additionally, we characterize the Gateaux and the Fr\'echet
differentiability of convex spectral functions in \cref{p:32}.
Finally, \cref{sec:5} establishes a constructive formula for the
Bregman proximity operators of (not necessarily convex) spectral
functions.

\section{Spectral decomposition systems}
\label{sec:2}

To motivate our definition, recall from the introduction that
a wide array of important optimization problems involve a function
$\Phi\colon\FH\to\RX$ defined on a real Hilbert space $\FH$ whose
values depend only on a certain spectral mapping
$\gamma\colon\FH\to\XX$ taking values in some simpler Euclidean
space $\XX$ in the sense that
\begin{equation}
\label{e:1}
(\forall X\in\FH)(\forall Y\in\FH)\quad
\gamma(X)=\gamma(Y)
\quad\Rightarrow\quad
\Phi(X)=\Phi(Y).
\end{equation}
In such applications, a central computational theme is to reduce
the evaluations of objects associated to $\Phi$ --
such as conjugate, subdifferential, and proximity operators --
to that of an associated, much simpler invariant function
$\varphi\colon\XX\to\RX$ and then lift the corresponding result
from $\XX$ to $\FH$ via a suitable embedding operator.
It would be illustrative to ground the discussion in a well-known
setting.
Consider the case where
$\FH=\SB^N$ is the space of $N\times N$ symmetric real matrices,
$\XX=\RR^N$, and $\gamma$ is the mapping that outputs the
eigenvalues of a matrix $X\in\FH$ in decreasing order.
In this setting, functions $\Phi$ satisfying \cref{e:1} are
precisely weakly unitarily invariant functions, and can be
represented as
\begin{equation}
\Phi=\varphi\circ\gamma,
\end{equation}
for some symmetric function $\varphi\colon\RR^N\to\RX$,
and the subdifferential of $\Phi$ is expressed as
(see \cref{n:1})
\begin{equation}
(\forall X\in\SB^N)\quad
\partial\Phi(X)=\menge{U(\Diag
y)U\tran}{y\in\partial\varphi\brk!{\gamma(X)},\,\,
U\in\US^N(\RR),\,\,X=U\brk!{\Diag\gamma(X)}U\tran}.
\end{equation}
This formula illustrates a central computational theme:
Finding an element in $\partial\Phi(X)$
is reduced to finding an element in $\partial\varphi(\gamma(X))$,
and lifting the result from $\RR^N$ to $\SB^N$
through a suitable embedding operator
\begin{equation}
\Lambda_U\colon\RR^N\to\SB^N\colon
x\mapsto U(\Diag x)U\tran
\end{equation}
that depends on a spectral decomposition of $X$.
Therefore, to be \emph{computationally relevant},
our framework must include the following structurally essential
ingredients, revealed through a closer inspection of the literature
cited in the introduction:
\begin{enumerate}
\item
A family $(\Lambda_{\ad})_{\ad\in\AD}$ of linear isometries
from $\XX$ to $\FH$ that allows for a ``spectral
decomposition'' of elements of $\FH$ through the spectral
mapping $\gamma$ in the sense that every $X\in\FH$ can be
decomposed as $X=\Lambda_{\ad}\gamma(X)$ for
some $\ad\in\AD$. For instance, in the motivating example
$\FH=\SB^N$ and $\XX=\RR^N$, this property is simply the eigenvalue
decomposition of matrices in $\FH$ given by
\begin{equation}
\Lambda_U\colon\XX\to\FH\colon
x\mapsto U(\Diag x)U\tran,
\end{equation}
where $U$ is an orthogonal matrix.
\item
A von~Neumann-type inequality relating the scalar product
of elements in $\FH$ to that of their spectra in $\XX$.
\end{enumerate}
It turns out that these must also satisfy some natural
invariance conditions with respect to the action of a certain set
$\SD$ on $\XX$ for the analysis; for instance, in the example of
symmetric matrices, $\SD$ is the set of permutation matrices.
Presenting a precise definition requires the introduction
of some notation. First, the scalar product and the associated norm
of a real Hilbert space are denoted by $\scal{\Cdot}{\Cdot}$ and
$\norm{\Cdot}$, respectively.
Next, the following notion of a set action
\cite[Chapitre~1, D\'{e}finition~{\S}3.1]{Bourbaki07}
on a Hilbert space will be central to our work.

\begin{definition}
\label{d:1}
Let $\HH$ be a real Hilbert space and let $\SD$ be a nonempty
set which acts on $\HH$ by linear isometries in the sense that
$\HH$ is endowed with a mapping
\begin{equation}
\SD\times\HH\to\HH\colon(\sd,x)\mapsto\sd\cdot x
\end{equation}
such that, for every $\sd\in\SD$,
$\HH\to\HH\colon x\mapsto\sd\cdot x$ is a linear isometry.
Then:
\begin{enumerate}
\item
\label{d:1i}
Given a set $\UU$, a mapping $\Phi\colon\HH\to\UU$ is said to be
\emph{$\SD$-invariant} if
\begin{equation}
(\forall x\in\HH)(\forall\sd\in\SD)\quad
\Phi(\sd\cdot x)=\Phi(x).
\end{equation}
\item
\label{d:1ii}
A subset $C$ of $\HH$ is said to be
\emph{$\SD$-invariant} if its indicator function
\begin{equation}
\iota_C\colon\HH\to\RX\colon x\mapsto
\begin{cases}
0&\text{if}\,\,x\in C,\\
\pinf&\text{otherwise},
\end{cases}
\end{equation}
is $\SD$-invariant or, equivalently,
\begin{equation}
(\forall x\in\HH)(\forall\sd\in\SD)\quad
x\in C
\quad\Leftrightarrow\quad
\sd\cdot x\in C.
\end{equation}
\end{enumerate}
\end{definition}

We now introduce the abstract framework we will work in.

\begin{definition}[spectral decomposition system]
\label{d:2}
Let $\FH$ and $\XX$ be real Hilbert spaces,
let $\SD$ be a nonempty set which acts on $\XX$ by linear
isometries, let $\gamma\colon\FH\to\XX$, and suppose that
$(\Lambda_{\ad})_{\ad\in\AD}$ is a family of linear
isometries from $\XX$ to $\FH$.
The tuple
$\mathfrak{S}=(\XX,\SD,\gamma,(\Lambda_{\ad})_{\ad\in\AD})$
is said to be a \emph{spectral decomposition system} for $\FH$ if
the following are satisfied:
\begin{enumerate}[label={\normalfont[\Alph*]}]
\item
\label{d:2a}
There exists an $\SD$-invariant mapping $\tau\colon\XX\to\XX$ such
that
\begin{equation}
\brk[s]!{\,
(\forall x\in\XX)(\exi\sd\in\SD)\,\,
x=\sd\cdot\tau(x)
\,}
\quad\text{and}\quad
\brk[s]!{\,
(\forall\ad\in\AD)\,\,
\gamma\circ\Lambda_{\ad}=\tau
\,}.
\end{equation}
\item
\label{d:2b}
$(\forall X\in\FH)(\exi\ad\in\AD)$
$X=\Lambda_{\ad}\gamma(X)$.
\item
\label{d:2c}
$(\forall X\in\FH)(\forall Y\in\FH)$
$\scal{X}{Y}\leq\scal{\gamma(X)}{\gamma(Y)}$.
\end{enumerate}
If $\mathfrak{S}=(\XX,\SD,\gamma,(\Lambda_{\ad})_{\ad\in\AD})$
is a spectral decomposition system for $\FH$, then:
\begin{itemize}
\item
The mapping $\gamma$ is called the \emph{spectral mapping} of the system
$\mathfrak{S}$.
\item
The mapping $\tau$ in property~\cref{d:2a} is called the
\emph{spectral-induced ordering mapping} of the system $\mathfrak{S}$.
(The motivation for this term will become apparent
from the concrete scenarios described in \cref{ex:3,ex:4,ex:5,ex:6}.)
\item
We set
\begin{equation}
\label{e:ax}
(\forall X\in\FH)\quad
\AD_X=\menge{\ad\in\AD}{X=\Lambda_{\ad}\gamma(X)}.
\end{equation}
Property~\cref{d:2b} ensures that the sets $(\AD_X)_{X\in\FH}$
are nonempty.
\item
Given $X\in\FH$, the vector $\gamma(X)$ is called the \emph{spectrum} of
$X$ with respect to $\mathfrak{S}$ and, for every $\ad\in\AD_X$,
the identity
\begin{equation}
X=\Lambda_{\ad}\gamma(X)
\end{equation}
is called a \emph{spectral decomposition} of $X$ with respect to
$\mathfrak{S}$.
\end{itemize}
\end{definition}

\begin{remark}
\label{r:6}
Let us discuss the connection between our framework and two
existing abstract frameworks.
\begin{enumerate}
\item
\label{r:6i}
The work \cite{Lewis96b} proposes the notion of a
normal decomposition system (see \cref{ex:2}) to unify the convex
analysis of spectral functions across several finite-dimensional
settings, such as Hermitian and rectangular matrices.
However, as demonstrated in \cite{Orlitzky22}, this framework
does not encompass the Euclidean Jordan algebra framework
of \cite{Baes07,Lourenco20,Sun08}
(see \cref{ex:3}) in general. As will be seen in \cref{ex:2,ex:3},
every normal decomposition system
and every Euclidean Jordan algebra induces a spectral decomposition
system. Our framework thereby extends the application scope of
\cite{Lewis96b}.
\item
\label{r:6ii}
The work \cite{Gowda19} proposes the notion
of a Fan--Theobald--von~Neumann system
as a wide-reaching framework to study optimization
problems involving spectral functions and spectral sets.
More precisely, a Fan--Theobald--von~Neumann system is a triple
$(\FH,\XX,\gamma)$ in which $\FH$ and $\XX$ are real Hilbert spaces,
and $\gamma\colon\FH\to\XX$ is a mapping that satisfies
\begin{equation}
(\forall X\in\FH)(\forall Y\in\FH)\quad
\max\menge{\scal{X}{Z}}{Z\in\FH\,\,\text{and}\,\,\gamma(Z)=\gamma(Y)}
=\scal{\gamma(X)}{\gamma(Y)}.
\end{equation}
However, its abstract nature allows only for the
\emph{characterization} of objects through a
\emph{nonconstructive} trace equality
\begin{equation}
\scal{X}{Y}=\scal{\gamma(X)}{\gamma(Y)}
\end{equation}
rather than providing an explicit construction.
While it is not immediately apparent from \cref{d:2},
it can be shown by using the results of \cref{sec:3}
that every spectral decomposition system
$(\XX,\SD,\gamma,(\Lambda_{\ad})_{\ad\in\AD})$ for a real Hilbert
space $\FH$ yields a Fan--Theobald--von~Neumann system
$(\FH,\XX,\gamma)$. Nevertheless,
the explicit incorporation of the embedding operators
$(\Lambda_{\ad})_{\ad\in\AD}$ and the spectral decomposition
property \cref{d:2b} in \cref{d:2}
circumvents the aforementioned computational limitation of the
Fan--Theobald--von~Neumann framework,
as discussed in \cref{G2};
see \cref{r:8}\,\cref{r:8ii} and \cref{r:5}\,\cref{r:5i}
for further discussions.
Finally, the examples in this section illustrate that our
framework is sufficiently general to encompass and strengthen
existing applications; see also
\cref{ex:86,ex:8bis-prox,ex:14}.
\end{enumerate}
\end{remark}

Having formally introduced the notion of a spectral decomposition
system, we now define and characterize the class of spectral
functions $\Phi\colon\FH\to\RXX$ within this framework, where
$\RXX=\intv{\minf}{\pinf}$. The following result shows that these
functions are precisely those of the form $\varphi\circ\gamma$,
where $\varphi\colon\XX\to\RXX$ is an $\SD$-invariant function.

\begin{proposition}
\label{p:5}
Let $\FH$ be a real Hilbert space,
let $\mathfrak{S}=(\XX,\SD,\gamma,(\Lambda_{\ad})_{\ad\in\AD})$
be a spectral decomposition system for $\FH$,
and let $\Phi\colon\FH\to\RXX$.
Then the following are equivalent:
\begin{enumerate}
\item
\label{p:5i}
$\Phi$ is a \emph{spectral function} in the sense that
\begin{equation}
\label{e:spec}
(\forall X\in\FH)(\forall Y\in\FH)\quad
\gamma(X)=\gamma(Y)
\quad\Rightarrow\quad
\Phi(X)=\Phi(Y).
\end{equation}
\item
\label{p:5ii}
$(\forall\ad\in\AD)(\forall\bd\in\AD)$
$\Phi\circ\Lambda_{\ad}=\Phi\circ\Lambda_{\bd}$.
\item
\label{p:5iii}
There exists an $\SD$-invariant function
$\varphi\colon\XX\to\RXX$ such that
$\Phi=\varphi\circ\gamma$.
\end{enumerate}
Moreover, if \cref{p:5iii} holds, then $(\forall\ad\in\AD)$
$\varphi=\Phi\circ\Lambda_{\ad}$.
\end{proposition}
\begin{proof}
Denote by $\tau$ the spectral-induced ordering mapping of the
system $\mathfrak{S}$. Then $\tau$ is $\SD$-invariant and
\begin{equation}
\label{e:14sd}
\brk[s]!{\,
(\forall x\in\XX)(\exi\sd\in\SD)\,\,
x=\sd\cdot\tau(x)
\,}
\quad\text{and}\quad
\brk[s]!{\,
(\forall\ad\in\AD)\,\,
\gamma\circ\Lambda_{\ad}=\tau
\,}.
\end{equation}

\cref{p:5i}\,$\Rightarrow$\,\cref{p:5ii}:
Let $\ad$ and $\bd$ be in $\AD$. For every $x\in\XX$,
since $\gamma(\Lambda_{\ad}x)=\tau(x)=\gamma(\Lambda_{\bd}x)$,
it results from \cref{e:spec} that
$\Phi(\Lambda_{\ad}x)=\Phi(\Lambda_{\bd}x)$.

\cref{p:5ii}\,$\Rightarrow$\,\cref{p:5iii}:
Fix $\bd\in\AD$ and set $\varphi=\Phi\circ\Lambda_{\bd}$. For every
$X\in\FH$ and every $\ad\in\AD_X$, we have
\begin{equation}
\varphi\brk!{\gamma(X)}
=(\Phi\circ\Lambda_{\bd})\brk!{\gamma(X)}
=(\Phi\circ\Lambda_{\ad})\brk!{\gamma(X)}
=\Phi\brk!{\Lambda_{\ad}\gamma(X)}
=\Phi(X),
\end{equation}
which confirms that $\Phi=\varphi\circ\gamma$.
In turn, we derive from \cref{e:14sd} that
\begin{equation}
\varphi
=\Phi\circ\Lambda_{\bd}
=(\varphi\circ\gamma)\circ\Lambda_{\bd}
=\varphi\circ(\gamma\circ\Lambda_{\bd})
=\varphi\circ\tau.
\end{equation}
Hence, the $\SD$-invariance of $\varphi$ follows from that of $\tau$.

\cref{p:5iii}\,$\Rightarrow$\,\cref{p:5i}:
Clear.

Finally, suppose that \cref{p:5iii} holds and take $\ad\in\AD$.
For every $x\in\XX$, upon choosing $\sd\in\SD$ such that
$x=\sd\cdot\tau(x)$,
we deduce from the identity $\gamma\circ\Lambda_{\ad}=\tau$
and the $\SD$-invariance of $\varphi$ that
$
\Phi(\Lambda_{\ad}x)
=(\varphi\circ\gamma)(\Lambda_{\ad}x)
=\varphi\brk{\gamma(\Lambda_{\ad}x)}
=\varphi\brk{\tau(x)}
=\varphi\brk{\sd\cdot\tau(x)}
=\varphi(x)$.
\end{proof}

In the light of \cref{p:5}, the following notion is well defined.

\begin{definition}
\label{d:3}
Let $\FH$ be a real Hilbert space,
let $\mathfrak{S}=(\XX,\SD,\gamma,(\Lambda_{\ad})_{\ad\in\AD})$
be a spectral decomposition system for $\FH$, and let
$\Phi\colon\FH\to\RXX$ be a spectral function in the sense that
\begin{equation}
(\forall X\in\FH)(\forall Y\in\FH)\quad
\gamma(X)=\gamma(Y)
\quad\Rightarrow\quad
\Phi(X)=\Phi(Y).
\end{equation}
The unique $\SD$-invariant function $\varphi\colon\XX\to\RXX$ such that
$\Phi=\varphi\circ\gamma$ is called the
\emph{invariant function associated with $\Phi$}.
\end{definition}

The remainder of this section is devoted to examples that
illustrate the breadth of the proposed notion of a
spectral decomposition system through a broad range of concrete
settings found in the literature. To this end, we first collect
some further notation that will be used in these examples, also in
the following sections.

\begin{notation}
  \label{n:1}
  Let $M$ and $N$ be strictly positive integers.
  \begin{itemize}
    \item
      $\KB$ denotes one of the following:
      the field $\RR$ of real numbers,
      the field $\CC$ of complex numbers,
      or the skew-field $\HB$ of Hamiltonian quaternions
      (we refer the reader to \cite{Rodman14} for background on
      quaternions).
    \item
      The canonical involution on $\KB$ is denoted by
      $\xi\mapsto\overline{\xi}$, which fixes only the elements of
      $\RR$, and the real part of $\xi\in\KB$ is
      $\Re\xi=(\xi+\overline{\xi})/2$.
    \item
      $\KB^{M\times N}$ denotes the standard real vector space of $M\times N$ matrices with entries in $\KB$.
    \item
      Given a matrix
      $X=\brk[s]{\xi_{ij}}_{\substack{1\leq i\leq M\\ 1\leq j\leq N}}
      \in\KB^{M\times N}$, its conjugate transpose is
      $X^*
      =\brk[s]{\overline{\xi_{ji}}}_{\substack{1\leq j\leq N\\ 1\leq i\leq M}}
      \in\KB^{N\times M}$.
    \item
      The trace of a matrix $X\in\KB^{N\times N}$ is denoted by $\tra X$.
    \item
      $\HS^N(\KB)$ denotes the vector subspace of $\KB^{N\times N}$ which consists of Hermitian matrices, that is,
      \begin{equation}
        \HS^N(\KB)=\menge{X\in\KB^{N\times N}}{X=X^*}.
      \end{equation}
    \item
      $\US^N(\KB)=\menge{U\in\KB^{N\times N}}{U^*U=\Id}$ is the set
      of unitary matrices.
    \item
      $\SO^N=\menge{U\in\US^N(\RR)}{\det U=1}$ is the special orthogonal group.
    \item
      $\PS_{\pm}^N$ denotes the multiplicative group of all matrices
      in $\RR^{N\times N}$ with entries in $\set{-1,0,1}$ which
      contain exactly one nonzero entry in every row and every column.
      (The elements of $\PS_{\pm}^N$ are called
      signed permutation matrices.)
    \item
      $\PS^N$ denotes the subgroup of $\PS_{\pm}^N$ which consists of matrices
      with entries in $\set{0,1}$.
      (The elements of $\PS^N$ are called permutation matrices.)
    \item
      For every $x=(\xi_i)_{1\leq i\leq N}\in\RR^N$,
      $x^{\downarrow}$ denotes the rearrangement vector of $x$ with
      entries listed in decreasing order, and
      $\abs{x}^{\downarrow}=
      (\abs{\xi_i})_{1\leq i\leq N}^{\downarrow}$.
  \end{itemize}
\end{notation}

The first example shows that every real Hilbert space admits at
least a spectral decomposition system.

\begin{example}
\label{ex:1}
Let $\HH$ be a nonzero real Hilbert space, let the set
$\set{-1,1}$ act on $\RR$ via multiplication, denote by $S$
the unit sphere of $\HH$, and set
\begin{equation}
(\forall e\in S)\quad
\Lambda_e\colon\RR\to\HH\colon\xi\mapsto\xi e.
\end{equation}
Then $(\RR,\set{-1,1},\norm{\Cdot},(\Lambda_e)_{e\in S})$
is a spectral decomposition system for $\HH$.
In this setting, spectral functions are precisely
\emph{radial} functions, that is, functions
$\Phi\colon\HH\to\RXX$ that satisfy
\begin{equation}
(\forall x\in\HH)(\forall y\in\HH)\quad
\norm{x}=\norm{y}
\quad\Rightarrow\quad
\Phi(x)=\Phi(y).
\end{equation}
\end{example}
\begin{proof}
It is clear that the operators $(\Lambda_e)_{e\in S}$ are linear
isometries, and that property~\cref{d:2a} in \cref{d:2} is
satisfied with $\tau=\abs{\Cdot}$.
Next, upon choosing $e\in S$, we infer from
\begin{equation}
(\forall x\in\HH)\quad
x=
\begin{cases}
\Lambda_{x/\norm{x}}\brk!{\norm{x}}&\text{if}\,\,x\neq 0,\\
\Lambda_e\brk!{\norm{x}}&\text{if}\,\,x=0,
\end{cases}
\end{equation}
that property~\cref{d:2b} in \cref{d:2} is satisfied.
Finally, we conclude the proof by recognizing that,
in the current setting, property~\cref{d:2c} in \cref{d:2} is
precisely the Cauchy--Schwarz inequality.
\end{proof}

The framework of \cref{ex:1} can be extended to handle
the class of block-radial functions, which is frequently
encountered in multi-task learning and block-structured,
mixed-norm regularization
\cite{Argyriou08,Combettes19,Gramfort12,Yuan06};
see also \cref{ex:8bis-prox}.

\begin{example}
\label{ex:8bis}
Let $(\HS_i)_{i\in I}$ be a family of nonzero real Hilbert
spaces, denote by
\begin{equation}
\HH=\bigoplus_{i\in I}\HS_i
\end{equation}
their Hilbert direct sum (see \cite[Example~2.1]{Livre1}), define
\begin{equation}
\label{e:pnom}
\gamma\colon\HH\to\ell^2(I)\colon
(\mathsf{u}_i)_{i\in I}\mapsto
\brk!{\norm{\mathsf{u}_i}_{\HS_i}}_{i\in I},
\end{equation}
let the set $S=\set{-1,1}^I$ act on $\ell^2(I)$ via
\begin{equation}
\brk!{(\sigma_i)_{i\in I},(\xi_i)_{i\in I}}\mapsto
(\sigma_i\xi_i)_{i\in I},
\end{equation}
and set
\begin{equation}
\AD=\Menge3{(\mathsf{u}_i)_{i\in I}\in\prod_{i\in I}\HS_i}{
(\forall i\in I)\,\,\norm{\mathsf{u}_i}_{\HS_i}=1}
\end{equation}
and
\begin{equation}
\brk!{\forall\ad=(\mathsf{u}_i)_{i\in I}\in\AD}\quad
\Lambda_{\ad}\colon\ell^2(I)\to\HH\colon
(\xi_i)_{i\in I}\mapsto(\xi_i\mathsf{u}_i)_{i\in I}.
\end{equation}
Then $\mathfrak{S}=(\ell^2(I),S,\gamma,
(\Lambda_{\ad})_{\ad\in\AD})$
is a spectral decomposition system for $\HH$.
In this setting, spectral functions are precisely
\emph{block-radial} functions, that is, functions
$\Phi\colon\HH\to\RXX$ that satisfy
\begin{equation}
\brk!{\forall u=(\mathsf{u}_i)_{i\in I}\in\HH}
\brk!{\forall v=(\mathsf{v}_i)_{i\in I}\in\HH}\quad
\brk[s]!{\,
(\forall i\in I)\,\,
\norm{\mathsf{u}_i}_{\HS_i}
=\norm{\mathsf{v}_i}_{\HS_i}
\,}
\quad\Rightarrow\quad
\Phi(u)=\Phi(v).
\end{equation}
\end{example}

We now consider a framework based on the Fourier transform
that underpins several phase retrieval problems
\cite{Bauschke02,Luke02,Youla82}; see also \cref{ex:86}.

\begin{example}[Fourier transform]
\label{ex:8}
If $\KB$ is $\RR$ or $\CC$, we
denote by $L^2(\RR;\KB)$ the set of (equivalence classes of
almost-everywhere equal) square-integrable
measurable functions $f\colon(\RR,\EuScript{L})\to(\KB,\BE(\KB))$,
where $\EuScript{L}$ and $\BE(\KB)$ stand for the Lebesgue
$\sigma$-algebra on $\RR$ and the Borel $\sigma$-algebra of $\KB$,
respectively.
Let $\HH$ be the real Hilbert space
obtained by equipping $L^2(\RR;\CC)$ with pointwise addition,
real-scalar multiplication, and the scalar product
\begin{equation}
(\forall f\in\HH)(\forall g\in\HH)\quad
\scal{f}{g}=\Re\int_{\RR}f(t)\overline{g(t)}\,dt,
\end{equation}
and let $\XX$ be the standard real Hilbert space $L^2(\RR;\RR)$.
Further, set
\begin{equation}
\AD=\menge{\theta\colon\RR\to\RR}{\text{$\theta$ is
$\brk!{\EuScript{L},\BE(\RR)}$-measurable}},
\end{equation}
let $\SX=\menge{\sigma\in\AD}{\sigma\in\set{-1,1}\,\,\text{a.e.}}$
act on $\XX$ via pointwise multiplication, and define
\begin{equation}
\label{e:pmft}
\gamma\colon\HH\to\XX\colon
f\mapsto\abs!{\widehat{f}}
\quad\text{and}\quad
(\forall\theta\in\AD)\,\,
\Lambda_{\theta}\colon\XX\to\HH\colon
x\mapsto\mathfrak{F}^{-1}\brk!{\ee^{\ii\theta}x},
\end{equation}
where
\begin{equation}
\mathfrak{F}\colon L^2(\RR;\CC)\to L^2(\RR;\CC)\colon
f\mapsto\widehat{f}
\end{equation}
denotes the Plancherel extension of the Fourier transform.
Then
$\mathfrak{S}=(\XX,\SX,\gamma,(\Lambda_{\theta})_{\theta\in\AD})$
is a spectral decomposition system for $\HH$.
In this setting, spectral functions are precisely
\emph{Fourier-phase-invariant} functions, that is, functions
$\Phi\colon\HH\to\RXX$ that satisfy
\begin{equation}
(\forall f\in\HH)(\forall g\in\HH)\quad
\abs!{\widehat{f}}=\abs!{\widehat{g}}\,\,\text{a.e.}
\quad\Rightarrow\quad
\Phi(f)=\Phi(g).
\end{equation}
\end{example}
\begin{proof}
For every $\theta\in\AD$, since the Plancherel theorem
\cite[Theorem~8.29]{Folland99} yields
\begin{equation}
(\forall x\in\XX)\quad
\norm{\Lambda_{\theta}x}^2
=\int_{\RR}\abs!{\mathfrak{F}^{-1}\brk!{\ee^{\ii\theta}x}(t)}^2\,dt
=\int_{\RR}\abs!{\ee^{\ii\theta(\omega)}x(\omega)}^2\,d\omega
=\int_{\RR}\abs{x(\omega)}^2\,d\omega
=\norm{x}^2,
\end{equation}
it follows that $\Lambda_{\theta}$ is a linear isometry.
To verify property~\cref{d:2a} in \cref{d:2}, set
$\tau\colon\XX\to\XX\colon x\mapsto\abs{x}$.
For every $x\in\XX$ and every $\sigma\in\SX$,
since $\abs{\sigma}=1$ a.e., we deduce that
$\tau\brk{\sigma x}
=\abs{\sigma x}
=\abs{x}
=\tau(x)$ and, in turn, that $\tau$ is $\SX$-invariant.
In addition, for every $x\in\XX$, upon setting
\begin{equation}
\sigma\colon\RR\to\RR\colon
\omega\mapsto
\begin{cases}
\sign\brk!{x(\omega)}&\text{if}\,\,x(\omega)\neq 0,\\
1&\text{if}\,\,x(\omega)=0,
\end{cases}
\end{equation}
we get $\sigma\in\SX$ and $x=\sigma\tau(x)$. Moreover,
\begin{equation}
(\forall\theta\in\AD)(\forall x\in\XX)\quad
\gamma(\Lambda_{\theta}x)
=\abs!{\mathfrak{F}\brk!{\mathfrak{F}^{-1}\brk!{\ee^{\ii\theta}x}}}
=\abs!{\ee^{\ii\theta}x}
=\abs{x}
=\tau(x).
\end{equation}
Next, for every $f\in\HH$, upon choosing
$\theta\in\AD$ such that
$\mathfrak{F}f=\ee^{\ii\theta}\abs{\mathfrak{F}f}$,
we obtain
\begin{equation}
\Lambda_{\theta}\gamma(f)
=\mathfrak{F}^{-1}\brk!{\ee^{\ii\theta}\abs{\mathfrak{F}f}}
=\mathfrak{F}^{-1}(\mathfrak{F}f)
=f,
\end{equation}
and property~\cref{d:2b} in \cref{d:2} is thus verified.
Finally, since $\mathfrak{F}$ is an orthogonal transformation
\cite[Theorem~8.29]{Folland99}, we deduce from the triangle
inequality \cite[Proposition~2.22]{Folland99} that
\begin{equation}
(\forall f\in\HH)(\forall g\in\HH)\quad
\scal{f}{g}
=\Re\int_{\RR}\widehat{f}(\omega)
\overline{\widehat{g}(\omega)}\,d\omega
\leq\int_{\RR}\abs!{\widehat{f}(\omega)}\,
\abs!{\widehat{g}(\omega)}\,d\omega
=\scal{\gamma(f)}{\gamma(g)},
\end{equation}
which completes the proof.
\end{proof}

The next example revolves around the decreasing rearrangements of
functions, a central notion in the theory of
rearrangement-invariant norms and optimization problems with
prescribed rearrangements \cite{Alvino89,Bennett88,Burton89}.

\begin{example}
\label{ex:7}
In this example, $\EuScript{L}$ and $\lambda$ designate the
Lebesgue $\sigma$-algebra and the Lebesgue measure on $\RR$,
respectively, and $\EuScript{L}_I$ denotes the restriction of
$\EuScript{L}$ on a set $I\in\EuScript{L}$.
Given a finite measure space $(\Omega,\FF,\mu)$
and using the notation $I=\intv{0}{\mu(\Omega)}$,
we denote by $\AD(\Omega,\mu)$ the set of all pairs
$(\sigma,T)$ in which
$\sigma\colon(\Omega,\FF)\to\RR$ is a measurable function such
that $\sigma\in\set{-1,1}$ $\mu$-a.e.,
and $T\colon(\Omega,\FF)\to(I,\EuScript{L}_I)$
is a \emph{measure-preserving} mapping,
that is, $T$ is a measurable mapping and satisfies
\begin{equation}
\brk!{\forall A\in\EuScript{L}_I}\quad
\mu\brk!{T^{-1}(A)}=\lambda(A);
\end{equation}
further, for every measurable function
$f\colon(\Omega,\FF)\to\RXX$,
its \emph{decreasing rearrangement} is
the measurable function
\begin{equation}
\abs{f}^{\downarrow}\colon(I,\EuScript{L}_I)\to\RPX\colon
t\mapsto\inf\menge{\alpha\in\RP}{
\mu_f(\alpha)\leq t},
\end{equation}
where
\begin{equation}
\mu_f\colon\RP\to I\colon\alpha\mapsto
\mu\brk2{\menge{\omega\in\Omega}{\abs{f(\omega)}>\alpha}}
\end{equation}
is the \emph{$\mu$-distribution} of $f$.
Now let $(\Omega,\FF,\mu)$ be a complete, finite, and nonatomic
measure space, set
$I=\intv{0}{\mu(\Omega)}$,
let
\begin{equation}
\HH=L^2\brk!{\Omega,\FF,\mu;\RXX}
\quad\text{and}\quad
\XX=L^2\brk!{I,\EuScript{L}_I,\lambda;\RXX},
\end{equation}
and define
\begin{equation}
\gamma\colon\HH\to\XX\colon f\mapsto\abs{f}^{\downarrow}.
\end{equation}
Moreover, let $\AD(I,\lambda)$ act on $\XX$ via
\begin{equation}
\label{e:yex7}
\brk!{(\kappa,Q),x}\mapsto\kappa(x\circ Q)
\end{equation}
and define
\begin{equation}
\label{e:kqlq}
\brk!{\forall(\sigma,T)\in\AD(\Omega,\mu)}\quad
\Lambda_{(\sigma,T)}\colon\XX\to\HH\colon
x\mapsto\sigma(x\circ T).
\end{equation}
Then $\mathfrak{S}=(\XX,\AD(I,\lambda),\gamma,
(\Lambda_{(\sigma,T)})_{(\sigma,T)\in\AD(\Omega,\mu)})$
is a spectral decomposition system for $\HH$.
In this setting, spectral functions are precisely
\emph{rearrangement-invariant} functions, that is,
functions $\Phi\colon\HH\to\RXX$ that satisfy
\begin{equation}
(\forall f\in\HH)(\forall g\in\HH)\quad
\mu_f=\mu_g
\quad\Rightarrow\quad
\Phi(f)=\Phi(g).
\end{equation}
(Two functions $f$ and $g$ in $\HH$ which have the same
distribution function are called \emph{equimeasurable}.)
\end{example}
\begin{proof}
First, since $(\Omega,\FF,\mu)$ is complete, finite, and nonatomic,
there exists a measure-preserving mapping from
$(\Omega,\FF)$ onto $(I,\EuScript{L}_I)$
\cite[Chapter~2, Proposition~7.4]{Bennett88},
which ensures that
$\AD(\Omega,\mu)\neq\emp$.
Second, for every measure-preserving mapping
$T\colon(\Omega,\FF)\to(I,\EuScript{L}_I)$
and every $(x,y)\in\XX\times\XX$, we derive from
\cite[Chapter~2, Propositions~1.8 and 7.2]{Bennett88} that
\begin{equation}
\label{e:1uby}
\int_{\Omega}\abs!{x\brk!{T(\omega)}}^2\,\mu(d\omega)
=2\int_0^{\pinf}\alpha\mu_{x\circ T}(\alpha)\,d\alpha
=2\int_0^{\pinf}\alpha\lambda_x(\alpha)\,d\alpha
=\int_I\abs{x(t)}^2\,dt
<\pinf,
\end{equation}
and verify at once that
\begin{equation}
x=y\,\,\text{$\lambda$-a.e.}
\quad\Rightarrow\quad
x\circ T=y\circ T\,\,\text{$\mu$-a.e.}
\end{equation}
and that
\begin{equation}
\label{e:5njw}
\abs{x\circ T}^{\downarrow}
=\abs{x}^{\downarrow}.
\end{equation}
Therefore, the action of $\AD(I,\lambda)$ on $\XX$,
as well as the operators
$(\Lambda_{(\sigma,T)})_{(\sigma,T)\in\AD(\Omega,\mu)}$,
are well defined;
in addition, it results from \cref{e:1uby} that
$\AD(I,\lambda)$ acts on $\XX$ by linear isometries,
and that the operators
$(\Lambda_{(\sigma,T)})_{(\sigma,T)\in\AD(\Omega,\mu)}$
are linear isometries.
Third, since \cite[Chapter~2, Proposition~1.8]{Bennett88} yields
\begin{equation}
(\forall f\in\HH)\quad
\int_I\brk!{\abs{f}^{\downarrow}(t)}^2\,dt
=\int_{\Omega}\abs{f(\omega)}^2\,\mu(d\omega)
<\pinf,
\end{equation}
we infer that $\gamma$ is well defined.
Next, to verify property~\cref{d:2a} in \cref{d:2},
define
\begin{equation}
\tau\colon\XX\to\XX\colon x\mapsto\abs{x}^{\downarrow}.
\end{equation}
It results from \cref{e:yex7,e:5njw} that
$\tau$ is $\AD(I,\lambda)$-invariant and that
\begin{equation}
\brk!{\forall(\sigma,T)\in\AD(\Omega,\mu)}(\forall x\in\XX)\quad
\gamma\brk!{\Lambda_{(\sigma,T)}x}
=\abs!{\sigma(x\circ T)}^{\downarrow}
=\abs{x\circ T}^{\downarrow}
=\tau(x),
\end{equation}
while Ryff's theorem \cite[Chapter~2, Theorem~7.5]{Bennett88}
asserts that, for every $x\in\XX$,
there exists $(\kappa,Q)\in\AD(I,\lambda)$ such that
$x=\kappa\abs{x}=\kappa(\abs{x}^{\downarrow}\circ Q)$
$\lambda$-a.e., which leads to
$x=(\kappa,Q)\cdot\tau(x)$.
An analogous argument using Ryff's theorem shows
that property~\cref{d:2b} in \cref{d:2} is satisfied.
Finally, property~\cref{d:2c} in \cref{d:2} follows from
the Hardy--Littlewood rearrangement inequality
\cite[Chapter~2, Theorem~2.2]{Bennett88}.
\end{proof}

The normal decomposition system framework of \cite{Lewis96b}
is the focus of the next example. This subsumes, in
particular, the semisimple real Lie algebra setting of
\cite{Lewis00,Tam98,Vincent97}.

\begin{example}[normal decomposition system]
  \label{ex:2}
  Let $(\HH,\mathscr{G},\gamma)$ be a \emph{normal decomposition
  system} in the sense of \cite[Definition~2.1]{Lewis96b}, that is,
  $\HH$ is a Euclidean space,
  $\mathscr{G}$ is a subgroup of the group of all orthogonal
  linear operators on $\HH$ (equipped with mapping composition),
  and $\gamma\colon\HH\to\HH$ is a $\mathscr{G}$-invariant mapping
  such that
  \begin{equation}
    \begin{cases}
      (\forall x\in\HH)(\exi T\in\mathscr{G})\,\,
      x=T\gamma(x)\\
      (\forall x\in\HH)(\forall y\in\HH)\,\,
      \scal{x}{y}\leq\scal{\gamma(x)}{\gamma(y)}.
    \end{cases}
  \end{equation}
  Additionally, let $\XX$ be a vector subspace of $\HH$ which
  contains the range of $\gamma$, set
  \begin{equation}
    \mathscr{G}_{\XX}=\menge{T\in\mathscr{G}}{T(\XX)=\XX},
  \end{equation}
  and let $\mathscr{G}_{\XX}$ act on $\XX$ via
  $(T,x)\mapsto Tx$.
  Suppose that
  \begin{equation}
    (\forall x\in\XX)(\exi T\in\mathscr{G}_{\XX})\quad
    x=T\gamma(x).
  \end{equation}
  Then $\mathfrak{S}=(\XX,\mathscr{G}_{\XX},\gamma,
  (\restr{T}{\XX})_{T\in\mathscr{G}})$
  is a spectral decomposition system for $\HH$.
  In this setting, spectral functions are precisely
  $\mathscr{G}$-invariant functions $\Phi\colon\HH\to\RXX$.
\end{example}
\begin{proof}
  In fact, property~\cref{d:2a} in \cref{d:2} is satisfied with
  $\tau=\restr{\gamma}{\XX}$.
\end{proof}

The following example concerns the Euclidean Jordan algebra
framework of \cite{Baes07,Lourenco20,Sendov07,Sun08},
which encompasses in particular the space $\HS^N(\KB)$ of Hermitian
matrices (see \cref{ex:4}).
As shown in \cite{Orlitzky22}, in general Euclidean Jordan algebras
cannot be embedded into a normal decomposition system of
\cref{ex:2}.

\begin{example}[Euclidean Jordan algebra]
  \label{ex:3}
  Let $\FH$ be a \emph{Euclidean Jordan algebra} (also known as \emph{formally
  real Jordan algebra}), that is, $\FH$ is a finite-dimensional real vector
  space which is endowed with a bilinear form
  \begin{equation}
    \FH\times\FH\to\FH\colon
    (X,Y)\mapsto X\circledast Y
  \end{equation}
  such that the following are satisfied:
  \begin{enumerate}[label={\normalfont[\Alph*]}]
    \item
      $(\forall X\in\FH)(\forall Y\in\FH)$
      $X\circledast Y=Y\circledast X$ and
      $X\circledast((X\circledast X)\circledast Y)
      =(X\circledast X)\circledast(X\circledast Y)$.
    \item
      There exists a scalar product
      $\killing{\Cdot}{\Cdot}$ on $\FH$ such that
      \begin{equation}
        (\forall X\in\FH)(\forall Y\in\FH)(\forall Z\in\FH)\quad
        \killing{X\circledast Y}{Z}=\killing{X}{Y\circledast Z}.
      \end{equation}
  \end{enumerate}
  (We refer to \cite{Faraut94} for background on and concrete examples of
  Euclidean Jordan algebras.)
  We equip $\FH$ with the trace scalar product
  \begin{equation}
    \label{e:wcyk}
    (\forall X\in\FH)(\forall Y\in\FH)\quad
    \scal{X}{Y}=\Tra(X\circledast Y).
  \end{equation}
  In addition, we denote by $E$ the identity element of $\FH$ and by $N$ the
  rank of $\FH$. Additionally, let $\PS^N$ act on $\RR^N$ via matrix-vector
  multiplication. Following \cite{Baes07}, we define a \emph{Jordan frame} of $\FH$ as a family
  $(A_i)_{1\leq i\leq N}$ in $\FH^N\smallsetminus\set{0}$ such that
  \begin{equation}
    \label{e:y7by}
    \begin{cases}
      \brk!{\forall i\in\set{1,\ldots,N}}
      \brk!{\forall j\in\set{1,\ldots,N}}\,\,
      A_i\circledast A_j=
      \begin{cases}
        A_i&\text{if}\,\,i=j,\\
        0&\text{if}\,\,i\neq j,
      \end{cases}\\
      \Sum_{i=1}^NA_i=E.
    \end{cases}
  \end{equation}
  (This is equivalent to the standard definition, e.g., on
  \cite[p.~44]{Faraut94}, of a complete system of orthogonal
  primitive idempotents satisfying \eqref{e:y7by} since such a
  system is necessarily of size $N$ and vice versa.)
  According to the spectral decomposition theorem for Euclidean Jordan algebras
  \cite[Theorem~III.1.2]{Faraut94}, for every $X\in\FH$, there exist a
  unique vector $(\lambda_1(X),\ldots,\lambda_N(X))\in\RR^N$, the entries of
  which are called the \emph{eigenvalues} of $X$, 
  and a Jordan frame $(A_i)_{1\leq i\leq N}$ such that
  \begin{equation}
    \lambda_1(X)\geq\cdots\geq\lambda_N(X)
    \quad\text{and}\quad
    X=\sum_{i=1}^N\lambda_i(X)A_i;
  \end{equation}
  this decomposition process thus defines a mapping
  \begin{equation}
    \label{e:9vmn}
    \lambda\colon\FH\to\RR^N\colon
    X\mapsto\brk!{\lambda_1(X),\ldots,\lambda_N(X)}.
  \end{equation}
  Further, let $\AD$ be the set of Jordan frames of $\FH$ and define
  \begin{equation}
    \label{e:cwpw}
    \brk!{\forall\ad=(A_i)_{1\leq i\leq N}\in\AD}\quad
    \Lambda_{\ad}\colon\RR^N\to\FH\colon
    x=(\xi_i)_{1\leq i\leq N}\mapsto\sum_{i=1}^N\xi_iA_i.
  \end{equation}
  Then $\mathfrak{S}=(\RR^N,\PS^N,\lambda,(\Lambda_{\ad})_{\ad\in\AD})$
  is a spectral decomposition system for $\FH$.
  In this setting, the class of spectral functions coincides with
  that studied in \cite{Baes07,Lourenco20,Sendov07,Sun08}.
\end{example}
\begin{proof}
  If $A$ is an element of a Jordan frame of $\FH$, then
  it follows from the spectral decomposition theorem that
  $\lambda(A)=(1,0,\ldots,0)$ and, in turn, from \cref{e:wcyk},
  \cref{e:y7by}, and \cite[Theorem~III.1.2]{Faraut94} that
  \begin{equation}
    \norm{A}^2
    =\Tra(A\circledast A)
    =\Tra A
    =\sum_{i=1}^N\lambda_i(A)
    =1.
  \end{equation}
  Therefore, for every $\ad=(A_i)_{1\leq i\leq N}\in\AD$, inasmuch as
  $(A_i)_{1\leq i\leq N}$ is an orthonormal family thanks to \cref{e:y7by}
  and \cref{e:wcyk}, we derive from \cref{e:cwpw} that
  \begin{equation}
    \brk!{\forall x=(\xi_i)_{1\leq i\leq N}\in\RR^N}\quad
    \norm{\Lambda_{\ad}x}^2
    =\sum_{i=1}^N\xi_i^2
    =\norm{x}_2^2.
  \end{equation}
  This implies that $\Lambda_{\ad}$ is a linear isometry. Next, define
  $\tau\colon\RR^N\to\RR^N\colon
  x\mapsto x^{\downarrow}$.
  It is clear that $\tau$ is $\PS^N$-invariant and
  $(\forall x\in\RR^N)(\exi P\in\PS^N)$
  $x=Px^{\downarrow}=P\tau(x)$.
  At the same time, we get from the spectral decomposition
  theorem that
  \begin{equation}
    \brk!{\forall\ad=(A_i)_{1\leq i\leq N}\in\AD}
    \brk!{\forall x=(\xi_i)_{1\leq i\leq N}\in\RR^N}\quad
    \lambda\brk{\Lambda_{\ad}x}
    =\lambda\brk3{\sum_{i=1}^N\xi_iA_i}
    =x^{\downarrow}
    =\tau(x)
  \end{equation}
  and that
  \begin{equation}
    (\forall X\in\FH)
    \brk!{\exi\ad=(A_i)_{1\leq i\leq N}\in\AD}\quad
    X
    =\sum_{i=1}^N\lambda_i(X)A_i
    =\Lambda_{\ad}\lambda(X),
  \end{equation}
  where the last equality follows from \cref{e:9vmn,e:cwpw}.
  Furthermore, we deduce from \cref{e:wcyk} and \cite[Theorem~23]{Baes07} that
  \begin{equation}
    (\forall X\in\FH)
    (\forall Y\in\FH)\quad
    \scal{X}{Y}
    \leq\sum_{i=1}^N\lambda_i(X)\lambda_i(Y)
    =\scal{\lambda(X)}{\lambda(Y)}.
  \end{equation}
  Altogether, $\mathfrak{S}$ is a spectral decomposition system for
  $\FH$. 
\end{proof}

In \cref{ex:3}, taking $\FH$ to be the Euclidean Jordan algebra of
Hermitian matrices (see \cite[Section~V.2]{Faraut94}), we obtain
the following framework.

\begin{example}[eigenvalue decomposition]
  \label{ex:4}
  Let $2\leq N\in\NN$. We equip $\HS^N(\KB)$ with the trace scalar
  product
  \begin{equation}
    \scal{\Cdot}{\Cdot}\colon
    (X,Y)\mapsto\Re\tra(XY)
  \end{equation}
  and let $\PS^N$ act on $\RR^N$ via matrix-vector multiplication.
  For every $X\in\HS^N(\KB)$, we denote by
  $\lambda(X)=(\lambda_1(X),\ldots,\lambda_N(X))$ the vector
  consisting of the $N$ eigenvalues (counting multiplicities) of
  $X$ listed in decreasing order;
  see \cite[Theorem~5.3.6(c)]{Rodman14} for the quaternion case.
  Additionally, set
  \begin{equation}
    \brk!{\forall U\in\US^N(\KB)}\quad
    \Lambda_U\colon\RR^N\to\HS^N(\KB)\colon
    x\mapsto U(\Diag x)U^*.
  \end{equation}
  Then
  $\mathfrak{S}=(\RR^N,\PS^N,\lambda,(\Lambda_U)_{U\in\US^N(\KB)})$
  is a spectral decomposition system for $\HS^N(\KB)$.
  In this setting, spectral functions are precisely
  \emph{weakly unitarily invariant} functions,
  that is, functions $\Phi\colon\HS^N(\KB)\to\RXX$ that satisfy
  \begin{equation}
    \brk!{\forall X\in\HS^N(\KB)}
    \brk!{\forall U\in\US^N(\KB)}\quad
    \Phi(UXU^*)=\Phi(X).
  \end{equation}
\end{example}

We now turn to a setting which has been extensively used in robust
principal component analysis and signal processing,
including applications involving quaternion matrices; see, for
instance, \cite{Candes10,Chan16,Chen19}.

\begin{example}[singular value decomposition]
  \label{ex:5}
  Let $M$ and $N$ be strictly positive integers
  and set $m=\min\set{M,N}$.
  Let $\FH$ be the Euclidean space obtained by equipping
  $\KB^{M\times N}$ with the trace scalar product
  \begin{equation}
    \label{e:gx9e}
    (X,Y)\mapsto\Re\tra(X^*Y),
  \end{equation}
  and let $\PS_{\pm}^m$ act on $\RR^m$ via matrix-vector multiplication.
  Given a matrix $X\in\FH$, the vector in
  $\RP^m$ formed by the $m$ singular values (counting
  multiplicities) of $X$,
  with the convention that they are listed in decreasing order, is denoted by
  $(\sigma_1(X),\ldots,\sigma_m(X))$; see
  \cite[Proposition~3.2.5\,(f)]{Rodman14} for
  singular value decomposition of matrices in $\HB^{M\times N}$. This thus
  defines a mapping
  \begin{equation}
    \sigma\colon\FH\to\RR^m\colon
    X\mapsto\brk!{\sigma_1(X),\ldots,\sigma_m(X)}.
  \end{equation}
  Further, set $\AD=\US^M(\KB)\times\US^N(\KB)$ and
  \begin{equation}
    \label{e:4l5f}
    \brk!{\forall\ad=(U,V)\in\AD}\quad
    \Lambda_{\ad}\colon\RR^m\to\FH\colon
    x\mapsto U(\Diag x)V^*,
  \end{equation}
  where the operator $\Diag\colon\RR^m\to\FH$ maps a vector
  $(\xi_i)_{1\leq i\leq m}$ to the diagonal matrix in $\FH$ of which the
  diagonal entries are $\xi_1,\ldots,\xi_m$. Then
  $\mathfrak{S}=(\RR^m,\PS_{\pm}^m,\sigma,(\Lambda_{\ad})_{\ad\in\AD})$
  is a spectral decomposition system for $\FH$.
  In this setting, spectral functions
  are precisely \emph{unitarily invariant} functions, that is,
  functions $\Phi\colon\FH\to\RXX$ that satisfy
  \begin{equation}
    (\forall X\in\FH)\brk!{\forall U\in\US^M(\KB)}
    \brk!{\forall V\in\US^N(\KB)}\quad
    \Phi(UXV)=\Phi(X).
  \end{equation}
\end{example}
\begin{proof}
  We derive from \cref{e:gx9e} and basic properties of the trace
  function (see, e.g.,
  \cite[properties~(c) and (d), p.~30]{Rodman14}
  and \cite[Proposition~V.2.1\,(i)]{Faraut94} for the case
  $\KB=\HB$) that the operators $(\Lambda_{\ad})_{\ad\in\AD}$ are
  linear isometries. To verify property~\cref{d:2a} in \cref{d:2},
  set $\tau\colon\RR^m\to\RR^m\colon x\mapsto\abs{x}^{\downarrow}$.
  It is evident that $\tau$ is $\PS_{\pm}^m$-invariant and
  $(\forall x\in\RR^m)(\exi P\in\PS_{\pm}^m)$
  $x=P\tau(x)$. In addition, it follows from \cref{e:4l5f} and the
  uniqueness of singular values (see
  \cite[Proposition~3.2.5\,(f)]{Rodman14} for the quaternion case)
  that
  \begin{align}
    \brk!{\forall\ad=(U,V)\in\AD}(\forall x\in\RR^m)\quad
    \sigma(\Lambda_{\ad}x)
    =\sigma\brk{\Diag x}
    =\abs{x}^{\downarrow}
    =\tau(x).
  \end{align}
  Finally, observe that, in the present setting, property~\cref{d:2b} in
  \cref{d:2} is precisely the singular value decomposition of matrices in
  $\FH$, and property~\cref{d:2c} in \cref{d:2} is precisely the von~Neumann
  trace inequality; see \cite[Theorem~7.4.1.1]{Horn13} for the real and
  complex cases, and \cite[Lemma~3]{Chan16} for the quaternion case.
\end{proof}

Our last example is a setting that arises in the study of isotropic
stored energy functions in nonlinear elasticity
\cite{Dacorogna07,Rosakis98}.

\begin{example}[signed singular value decomposition]
  \label{ex:6}
  Let $2\leq N\in\NN$ and let $\FH$ be the Euclidean space obtained by
  equipping $\RR^{N\times N}$ with the trace scalar product
  \begin{equation}
    (X,Y)\mapsto\tra(X\tran Y),
  \end{equation}
  let $\SD$ be the subgroup of $\PS_{\pm}^N$ which consists of all matrices
  with an even number of entries equal to $-1$, and let $\SD$ act on $\RR^N$
  via matrix-vector multiplication. As in \cref{ex:5},
  $\sigma(X)=(\sigma_1(X),\ldots,\sigma_N(X))$ designates the vector of the
  $N$ singular values of a matrix $X\in\FH$, with the convention that
  $\sigma_1(X)\geq\cdots\geq\sigma_N(X)$. Define a mapping
  \begin{equation}
    \gamma\colon\FH\to\RR^N\colon
    X\mapsto\brk!{\gamma_1(X),\ldots,\gamma_N(X)}
  \end{equation}
  by
  \begin{equation}
    (\forall X\in\FH)
    \brk!{\forall i\in\set{1,\ldots,N}}\quad
    \gamma_i(X)=
    \begin{cases}
      \sigma_i(X)&\text{if}\,\,1\leq i\leq N-1,\\
      \sigma_N(X)\sign(\det X)&\text{if}\,\,i=N.
    \end{cases}
  \end{equation}
  Finally, set $\AD=\SO^N\times\SO^N$ and
  \begin{equation}
    \brk!{\forall\ad=(U,V)\in\AD}\quad
    \Lambda_{\ad}\colon\RR^N\to\FH\colon
    x\mapsto U(\Diag x)V\tran.
  \end{equation}
  Then $\mathfrak{S}=(\RR^N,\SD,\gamma,(\Lambda_{\ad})_{\ad\in\AD})$
  is a spectral decomposition system for $\FH$.
  In this setting, spectral functions are precisely
  \emph{$\SO^N\times\SO^N$-invariant} functions, that is, functions
  $\Phi\colon\FH\to\RXX$ that satisfy
  \begin{equation}
    (\forall X\in\FH)
    \brk!{\forall U\in\SO^N}
    \brk!{\forall V\in\SO^N}\quad
    \Phi(UXV)=\Phi(X).
  \end{equation}
\end{example}
\begin{proof}
  It is clear that the operators $(\Lambda_{\ad})_{\ad\in\AD}$ are linear
  isometries. Next, to verify property~\cref{d:2a} in \cref{d:2}, we define a
  mapping $\tau\colon\RR^N\to\RR^N$ as follows:
  for every $x=(\xi_i)_{1\leq i\leq N}\in\RR^N$,
  let $(\mu_i)_{1\leq i\leq N}=\abs{x}^{\downarrow}$ and set
  $\tau(x)=(\nu_i)_{1\leq i\leq N}$, where
  \begin{equation}
    \brk!{\forall i\in\set{1,\ldots,N}}\quad
    \nu_i=
    \begin{cases}
      \mu_i&\text{if}\,\,1\leq i\leq N-1,\\
      \mu_N\sign(\xi_1\cdots\xi_N)&\text{if}\,\,i=N.
    \end{cases}
  \end{equation}
  One can verify that $\tau$ is $\SD$-invariant and
  $(\forall x\in\RR^N)(\exi P\in\SD)$ $x=P\tau(x)$.
  The condition
  \begin{equation}
    (\forall\ad\in\AD)\quad\gamma\circ\Lambda_{\ad}=\tau
  \end{equation}
  in property~\cref{d:2a} in \cref{d:2} follows from the orthogonal invariance
  of $\sigma$ and the construction of $\gamma$.
  Moreover, in the present setting, property~\cref{d:2b} in \cref{d:2} reads
  \begin{equation}
    \brk!{\forall X\in\RR^{N\times N}}
    (\exi U\in\SO^N)(\exi V\in\SO^N)\quad
    X=U\brk!{\Diag\gamma(X)}V\tran,
  \end{equation}
  which is an easy consequence of the singular value decomposition theorem.
  Finally, property~\cref{d:2c} in \cref{d:2} follows from
  \cite[Lemma~1.3]{Rosakis98}.
\end{proof}

\section{Properties of spectral and spectral-induced ordering mappings}
\label{sec:3}

In this section, we present properties of spectral and
spectral-induced ordering mappings that will be frequently employed
in our analysis.
More importantly, we provide in \cref{p:3} a necessary and
sufficient condition -- formulated in terms of a ``simultaneous
spectral decomposition'' -- for equality in property \cref{d:2c} in
\cref{d:2}, which will serve as an essential tool
in \cref{sec:4,sec:5} for explicit constructions
of several convex analytical objects associated with spectral
functions via their associated invariant
functions and the spectral decomposition property \cref{d:2b} in
\cref{d:2}.

Throughout this paper, we shall operate under the
umbrella of the following setting.

\begin{assumption}
\label{a:1}
$\FH$ is a real Hilbert space,
$\mathfrak{S}=(\XX,\SD,\gamma,(\Lambda_{\ad})_{\ad\in\AD})$
is a spectral decomposition system for $\FH$,
and $\tau\colon\XX\to\XX$ is the spectral-induced ordering
mapping of the system $\mathfrak{S}$, that is, $\tau$ is
$\SD$-invariant and satisfies
\begin{equation}
\label{e:1ae1}
(\forall x\in\XX)(\exi\sd\in\SD)\quad
x=\sd\cdot\tau(x)
\end{equation}
and
\begin{equation}
\label{e:1ae2}
(\forall\ad\in\AD)\quad\gamma\circ\Lambda_{\ad}=\tau.
\end{equation}
\end{assumption}

The following standard fact in functional analysis will be
repeatedly used, and we include a proof for completeness.

\begin{lemma}
\label{l:1-}
Let $\HH$ and $\GG$ be real Hilbert spaces
and suppose that $L\colon\HH\to\GG$ is a linear isometry.
Then the following hold:
\begin{enumerate}
\item
\label{l:1-i}
$L^*\circ L=\Id_{\HH}$.
\item
\label{l:1-ii}
$(\forall x\in\HH)(\forall y\in\HH)$
$\scal{x}{y}=\scal{Lx}{Ly}$.
\end{enumerate}
\end{lemma}
\begin{proof}
Let $x$ and $y$ be in $\HH$.

\cref{l:1-i}:
Since
$\norm{x}^2-2\scal{x}{y}+\norm{y}^2
=\norm{x-y}^2
=\norm{Lx-Ly}^2
=\norm{Lx}^2-2\scal{Lx}{Ly}+\norm{Ly}^2
=\norm{x}^2-2\scal{L^*(Lx)}{y}+\norm{y}^2$,
we get $\scal{L^*(Lx)}{y}=\scal{x}{y}$.
Hence, since $y$ is arbitrary, it follows that $L^*(Lx)=x$.

\cref{l:1-ii}:
This follows from \cref{l:1-i}.
\end{proof}

We now establish several basic properties of spectral-induced
ordering mappings.

\begin{proposition}
\label{p:1}
Suppose that \cref{a:1} is in force. Then the following hold:
\begin{enumerate}
\item
\label{p:1i}
$\tau$ is idempotent, that is, $\tau\circ\tau=\tau$.
\item
\label{p:1ii}
$(\forall x\in\XX)(\forall\ad\in\AD)$
$\norm{\Lambda_{\ad}x}=\norm{x}=\norm{\tau(x)}$.
\item
\label{p:1iii}
$(\forall x\in\XX)(\forall y\in\XX)$
$\scal{x}{y}\leq\scal{\tau(x)}{\tau(y)}$.
\item
\label{p:1iv}
$\tau$ is nonexpansive, that is,
$(\forall x\in\XX)(\forall y\in\XX)$
$\norm{\tau(x)-\tau(y)}\leq\norm{x-y}$.
\item
\label{p:1v}
$\ran\tau$ is a closed convex cone in $\XX$.
\end{enumerate}
\end{proposition}
\begin{proof}
\cref{p:1i}:
A consequence of \cref{e:1ae1} and the $\SD$-invariance of $\tau$.

\cref{p:1ii}:
This follows from \cref{e:1ae1} and the assumption that the
operators $(\Lambda_{\ad})_{\ad\in\AD}$ are linear isometries and
that $\SD$ acts on $\XX$ by linear isometries.

\cref{p:1iii}:
Take $\ad\in\AD$. We derive from
property~\cref{d:2c} in \cref{d:2} and \cref{e:1ae2} that
$(\forall x\in\XX)(\forall y\in\XX)$
$\scal{x}{y}
=\scal{\Lambda_{\ad}x}{\Lambda_{\ad}y}
\leq\scal{\gamma(\Lambda_{\ad}x)}{\gamma(\Lambda_{\ad}y)}
=\scal{\tau(x)}{\tau(y)}$.

\cref{p:1iv}:
Combine \cref{p:1iii,p:1ii}.

\cref{p:1v}:
We employ the technique of \cite[Theorem~2.4]{Lewis96b},
which establishes this result in the setting of \cref{ex:2}.
By \cref{p:1i},
$(\forall y\in\ran\tau)$
$y=\tau(y)$.
Thus, for every $x\in\XX$, it follows from \cref{p:1iii} that
\begin{equation}
\sup_{y\in\ran\tau}\scal{y}{x-\tau(x)}
=\sup_{y\in\ran\tau}\brk!{\scal{y}{x}-\scal{\tau(y)}{\tau(x)}}
\leq 0
\end{equation}
and, therefore, that
$x-\tau(x)\in(\ran\tau)^{\ominus}$,
where $K^{\ominus}=\menge{u\in\XX}{\sup\scal{K}{u}\leq 0}$
denotes the polar cone of a subset $K$ of $\XX$.
In turn, for every $x\in(\ran\tau)^{\ominus\ominus}$,
we derive from \cref{p:1ii} that
\begin{equation}
0\geq
2\scal{x}{x-\tau(x)}
=\norm{x}^2+\norm{x-\tau(x)}^2-\norm{\tau(x)}^2
=\norm{x-\tau(x)}^2,
\end{equation}
which implies that $x=\tau(x)\in\ran\tau$.
Hence $(\ran\tau)^{\ominus\ominus}\subset\ran\tau$.
Thus, because $(\ran\tau)^{\ominus\ominus}$ is the smallest
closed convex cone in $\XX$ that contains
$\ran\tau$ \cite[Propositions~6.24\,(iii) and 6.33]{Livre1},
we must have $(\ran\tau)^{\ominus\ominus}=\ran\tau$.
Consequently, $\ran\tau$ is a closed convex cone.
\end{proof}

\cref{p:1}\,\cref{p:1iii} facilitates the construction
of another spectral decomposition system.

\begin{example}
\label{ex:9}
Suppose that \cref{a:1} is in force and define
\begin{equation}
(\forall\sd\in\SD)\quad
\pi_{\sd}\colon\XX\to\XX\colon x\mapsto\sd\cdot x.
\end{equation}
Then $(\XX,\SD,\tau,(\pi_{\sd})_{\sd\in\SD})$ is a spectral
decomposition system for $\XX$.
\end{example}

Analogous to \cref{p:1}, the following result provides basic
properties of spectral mappings.

\begin{proposition}
  \label{p:2}
  Suppose that \cref{a:1} is in force. Then the following hold:
  \begin{enumerate}
    \item
      \label{p:2i}
      $\tau\circ\gamma=\gamma$ and $(\forall\ad\in\AD)$
      $\gamma\circ\Lambda_{\ad}\circ\gamma=\gamma$.
    \item
      \label{p:2ii}
      $\ran\gamma=\ran\tau$ is a closed convex cone in $\XX$.
    \item
      \label{p:2iii}
      $(\forall X\in\FH)(\forall\ad\in\AD)$
      $\norm{X}
      =\norm{\gamma(X)}
      =\norm{\Lambda_{\ad}\gamma(X)}$.
    \item
      \label{p:2iv}
      $\gamma$ is nonexpansive, that is,
      $(\forall X\in\FH)(\forall Y\in\FH)$
      $\norm{\gamma(X)-\gamma(Y)}\leq\norm{X-Y}$.
    \item
      \label{p:2v}
      $(\forall X\in\FH)(\forall\alpha\in\RP)$
      $\gamma(\alpha X)=\alpha\gamma(X)$.
  \end{enumerate}
\end{proposition}
\begin{proof}
  \cref{p:2i}:
  In view of \cref{e:1ae2}, it is enough to show that
  $\tau\circ\gamma=\gamma$. Take $X\in\FH$ and $\ad\in\AD_X$.
  By \cref{e:1ae2}, $\tau=\gamma\circ\Lambda_{\ad}$. However, by
  the very definition of $\AD_X$, $X=\Lambda_{\ad}\gamma(X)$. Hence
  $\tau\brk{\gamma(X)}
  =(\gamma\circ\Lambda_{\ad})\brk{\gamma(X)}
  =\gamma\brk{\Lambda_{\ad}\gamma(X)}
  =\gamma(X)$.

  \cref{p:2ii}:
  It follows from \cref{p:2i} that $\ran\gamma\subset\ran\tau$, while
  \cref{e:1ae2} yields $\ran\tau\subset\ran\gamma$.
  Thus $\ran\gamma=\ran\tau$, which is a closed convex cone in $\XX$
  thanks to \cref{p:1}\,\cref{p:1v}.

  \cref{p:2iii}:
  This follows from the assumption that the operators
  $(\Lambda_{\ad})_{\ad\in\AD}$ are linear isometries,
  and from the spectral decomposition of elements in $\FH$;
  see property~\cref{d:2b} in \cref{d:2}.

  \cref{p:2iv}:
  Combine property~\cref{d:2c} in \cref{d:2} and \cref{p:2iii}.

  \cref{p:2v}:
  For every $X\in\FH$ and every $\alpha\in\RP$, we derive from
  \cref{p:2iii} and property~\cref{d:2c} in \cref{d:2} that
  \begin{align}
    \norm{\gamma(\alpha X)-\alpha\gamma(X)}^2
    &=\norm{\gamma(\alpha X)}^2-2\alpha\scal{\gamma(\alpha X)}{
      \gamma(X)}+\alpha^2\norm{\gamma(X)}^2
    \nonumber\\
    &\leq\norm{\alpha X}^2-2\alpha\scal{\alpha X}{X}+\alpha^2\norm{X}^2
    \nonumber\\
    &=0,
  \end{align}
  which leads to $\gamma(\alpha X)=\alpha\gamma(X)$.
\end{proof}

In the convex analysis of spectral functions
in the settings of Euclidean Jordan algebras and matrices
(see \cref{ex:3,ex:4,ex:5}),
the equality case
\begin{equation}
\label{e:eqvn}
\scal{X}{Y}=\scal{\gamma(X)}{\gamma(Y)}
\end{equation}
of the von~Neumann-type trace inequality
is a key condition in the derivations
of formulae to evaluate the
subdifferentials and the proximity operators of spectral functions
via the associated invariant functions.
To obtain constructive formulae -- suitable e.g.\
for algorithmic implementations --
it is crucial to establish a characterization -- in
terms of a ``simultaneous spectral decomposition'' -- for this
equality; see
\cite{Baes07,Livre1,Beck17,Lewis95,Lewis96a,Mordukhovich22}
and the discussions in \cref{sec:4,sec:5}.
In view of this and to facilitate the derivations of results
in \cref{sec:4,sec:5},
we now leverage the spectral decomposition property~\cref{d:2b} in
\cref{d:2} to establish a necessary and sufficient condition
for \cref{e:eqvn} in our abstract framework.

\begin{proposition}
\label{p:3}
Suppose that \cref{a:1} is in force and
let $\DD$ be a nonempty subset of $\FH$ such that
\begin{equation}
\label{e:h9hb}
\dim\aff(\DD-\DD)<\pinf.
\end{equation}
Then
\begin{equation}
\brk[s]!{\,(\forall X\in\DD)(\forall Y\in\DD)\,\,
\scal{X}{Y}=\scal{\gamma(X)}{\gamma(Y)}\,}
\quad\Leftrightarrow\quad
\brk[s]!{\,
(\exi\ad\in\AD)(\forall X\in\DD)\,\,
X=\Lambda_{\ad}\gamma(X)\,}.
\end{equation}
\end{proposition}
\begin{proof}
We employ the technique of \cite[Theorem~2.2]{Lewis96b},
which establishes the result in the setting of \cref{ex:2}.
In the light of \cref{e:h9hb}, the relative algebraic
interior of $\conv\DD$ is nonempty
\cite[property~(vii) on p.~3]{Zalinescu2002},
that is, there exists $Z\in\conv\DD$ such that
\begin{equation}
\label{e:7ud2}
\brk!{\forall Y\in\aff(\DD-\DD)}
\brk!{\exi\delta\in\zeroun}
\brk!{\forall\alpha\in\intv{0}{\delta}}\quad
Z+\alpha Y\in\conv\DD.
\end{equation}
We now establish the implication ``$\Rightarrow$'' by contradiction.
Suppose that
\begin{equation}
\label{e:agi1}
(\forall X\in\DD)(\forall Y\in\DD)\quad
\scal{X}{Y}=\scal{\gamma(X)}{\gamma(Y)},
\end{equation}
take $\ad\in\AD_Z$, and assume that there exists $X\in\DD$ for which
$X\neq\Lambda_{\ad}\gamma(X)$. \cref{p:2}\,\cref{p:2iii} entails
that
\begin{equation}
2\scal{X}{\Lambda_{\ad}\gamma(X)}
=\norm{X}^2
+\norm{\Lambda_{\ad}\gamma(X)}^2
-\norm{X-\Lambda_{\ad}\gamma(X)}^2
=2\norm{X}^2-\norm{X-\Lambda_{\ad}\gamma(X)}^2
<2\norm{X}^2.
\end{equation}
On the other hand,
using \cref{e:7ud2} with $Y=Z-X\in\aff(\DD-\DD)$,
we deduce that there exists $\delta\in\zeroun$ such that
$Z+(\delta/2)(Z-X)\in\conv\DD$,
which implies that there exist families
$(\alpha_i)_{0\leq i\leq n}$ in $\zeroun$ and
$(X_i)_{1\leq i\leq n}$ in $\DD$ such that $\sum_{i=0}^n\alpha_i=1$
and $Z=\alpha_0X+\sum_{i=1}^n\alpha_iX_i$.
Therefore, we derive from property~\cref{d:2c} in \cref{d:2},
\cref{p:2}\,\cref{p:2i}, and \cref{e:agi1} that
\begin{align}
\scal{\gamma(Z)}{\gamma(X)}
&=\scal{\Lambda_{\ad}\gamma(Z)}{\Lambda_{\ad}\gamma(X)}
\nonumber\\
&=\scal{Z}{\Lambda_{\ad}\gamma(X)}
\nonumber\\
&=\alpha_0\scal{X}{\Lambda_{\ad}\gamma(X)}
+\sum_{i=1}^n\alpha_i\scal{X_i}{\Lambda_{\ad}\gamma(X)}
\nonumber\\
&<\alpha_0\norm{X}^2
+\sum_{i=1}^n\alpha_i
\scal!{\gamma(X_i)}{\gamma\brk!{\Lambda_{\ad}\gamma(X)}}
\nonumber\\
&=\alpha_0\norm{X}^2+\sum_{i=1}^n\alpha_i\scal{\gamma(X_i)}{\gamma(X)}
\nonumber\\
&=\alpha_0\norm{X}^2+\sum_{i=1}^n\alpha_i\scal{X_i}{X}
\nonumber\\
&=\scal{Z}{X}
\nonumber\\
&\leq\scal{\gamma(Z)}{\gamma(X)},
\end{align}
which is impossible. The reverse implication follows from the
assumption that the operators $(\Lambda_{\ad})_{\ad\in\AD}$
are linear isometries.
\end{proof}

\begin{remark}
\label{r:3}
Let us comment on \cref{p:3}.
\begin{enumerate}
\item
Condition~\cref{e:h9hb} holds, in particular,
when $\FH$ is finite-dimensional or when $\DD$ is a finite set.
\item
\cref{p:3} specializes to \cite[Theorem~2.2]{Lewis96b} in the
context of normal decomposition systems of \cref{ex:2},
and to \cite[Theorem~4.1]{Gowda22} in the context of Euclidean
Jordan algebras of \cref{ex:3}.
\item
In the settings of \cref{ex:4,ex:5,ex:6}, where $\KB$ is
$\RR$ or $\CC$, and with the choice $\DD=\set{X,Y}$,
we recover from \cref{p:3} classical conditions for
equality in von~Neumann-type trace inequalities;
see \cite{Lewis05a,Rosakis98,Theobald75,Neumann37}.
\end{enumerate}
\end{remark}

We close this section with some basic properties of invariant
functions.

\begin{proposition}
\label{p:6}
Suppose that \cref{a:1} is in force and let
$\varphi\colon\XX\to\RXX$. Then the following hold:
\begin{enumerate}
\item
\label{p:6i}
$\varphi$ is $\SD$-invariant
\,$\Leftrightarrow$\,
$\varphi\circ\tau=\varphi$
\,$\Leftrightarrow$\,
$(\forall\ad\in\AD)$
$\varphi\circ\gamma\circ\Lambda_{\ad}=\varphi$.
\item
\label{p:6ii}
Suppose that $\varphi$ is $\SD$-invariant. Then the following are
satisfied:
\begin{enumerate}
\item
\label{p:6iia}
$\intdom(\varphi\circ\gamma)=\gamma^{-1}(\intdom\varphi)$.
\item
\label{p:6iib}
$\intdom(\varphi\circ\gamma)\neq\emp$
\,$\Leftrightarrow$\,
$\intdom\varphi\neq\emp$.
\end{enumerate}
\item
\label{p:6iii}
Let $\psi\colon\XX\to\RXX$ and suppose that
$\varphi$ and $\psi$ are $\SD$-invariant.
Then $\varphi=\psi$ if and only if
$\restr{\varphi}{\ran\gamma}=\restr{\psi}{\ran\gamma}$.
\end{enumerate}
\end{proposition}
\begin{proof}
\cref{p:6i}:
Use \cref{e:1ae1}, the $\SD$-invariance of $\tau$, and
\cref{e:1ae2}.

\cref{p:6iia}:
Suppose that $X\in\intdom(\varphi\circ\gamma)$ and let
$\varepsilon\in\zeroun$ be such that the closed ball
$B(X;\varepsilon)$ is contained in $\dom(\varphi\circ\gamma)$.
We must show that $\gamma(X)\in\intdom\varphi$.
To do so, let $y\in B(\gamma(X);\varepsilon)$ and $\ad\in\AD_X$.
Then
$\norm{\Lambda_{\ad}y-X}
=\norm{\Lambda_{\ad}y-\Lambda_{\ad}\gamma(X)}
=\norm{y-\gamma(X)}
\leq\varepsilon$, which implies that
$\Lambda_{\ad}y\in B(X;\varepsilon)\subset\dom(\varphi\circ\gamma)$.
Thus, we get from \cref{p:6i} that
$\varphi(y)=(\varphi\circ\gamma)(\Lambda_{\ad}y)<\pinf$.
Hence $B(\gamma(X);\varepsilon)\subset\dom\varphi$ and, therefore,
$\gamma(X)\in\intdom\varphi$.
The reverse inclusion is proved similarly.

\cref{p:6iib}:
By \cref{p:6iia},
$\intdom(\varphi\circ\gamma)\neq\emp$
\,$\Rightarrow$\,
$\intdom\varphi\neq\emp$.
To establish the reverse implication,
assume that $\intdom\varphi\neq\emp$ and let us construct
a point $X\in\FH$ such that $X\in\intdom(\varphi\circ\gamma)$
or, equivalently, thanks to \cref{p:6iia},
$\gamma(X)\in\intdom\varphi$.
Towards this goal, take $x\in\intdom\varphi$,
let $\sd\in\SD$ be such that
$x=\sd\cdot\tau(x)$, and let
$\varepsilon\in\zeroun$ be such that
$B(x;\varepsilon)\subset\dom\varphi$.
For every $y\in B(\tau(x);\varepsilon)$,
since $\SD$ acts on $\XX$ by linear isometries,
we get
$\norm{\sd\cdot y-x}
=\norm{\sd\cdot y-\sd\cdot\tau(x)}
=\norm{y-\tau(x)}
\leq\varepsilon$ and, in turn,
$\sd\cdot y\in B(x;\varepsilon)\subset\dom\varphi$,
from which and the $\SD$-invariance of $\varphi$ we infer that
$\varphi(y)=\varphi(\sd\cdot y)<\pinf$.
Hence $B(\tau(x);\varepsilon)\subset\dom\varphi$ and we have thus
shown that $\tau(x)\in\intdom\varphi$.
Now take $\ad\in\AD$ and set $X=\Lambda_{\ad}x$.
By \cref{e:1ae2}, $\gamma(X)=\tau(x)\in\intdom\varphi$.

\cref{p:6iii}:
Combine \cref{p:6i} and \cref{p:2}\,\cref{p:2ii}.
\end{proof}

\section{Conjugation and subdifferentiability of spectral functions}
\label{sec:4}

Let $\HH$ be a real Hilbert space.
A function $f\colon\HH\to\RXX$ is said to be convex if
\begin{equation}
  (\forall x\in\dom f)(\forall y\in\dom f)
  \brk!{\forall\alpha\in\zeroun}\quad
  f\brk!{\alpha x+(1-\alpha)y}\leq\alpha f(x)+(1-\alpha)f(y),
\end{equation}
where $\dom f=\menge{x\in\HH}{f(x)<\pinf}$ is the domain of $f$,
and $f$ is said to be proper if $\minf\notin f(\HH)$ and
$\dom f\neq\emp$.
We denote by $\Gamma_0(\HH)$ the set of proper lower semicontinuous
convex functions from $\HH$ to $\RX$.

Two fundamental convex analytical objects attached to a proper
function $f\colon\HH\to\RX$ are its conjugate
\begin{equation}
  \label{e:lege}
  f^*\colon\HH\to\RXX\colon
  y\mapsto\sup_{x\in\HH}\brk!{\scal{x}{y}-f(x)},
\end{equation}
and its subdifferential
\begin{equation}
  \partial f\colon\HH\to 2^{\HH}\colon
  x\mapsto\menge{y\in\HH}{(\forall z\in\HH)\,\,\scal{z-x}{y}+f(x)\leq f(z)},
\end{equation}
which are related by the Fenchel--Young equality
\cite[Proposition~16.10]{Livre1}
\begin{equation}
  \label{e:vpdz}
  (\forall x\in\HH)(\forall y\in\HH)\quad
  y\in\partial f(x)
  \,\,\Leftrightarrow\,\,
  f(x)+f^*(y)=\scal{x}{y}
  \,\,\Leftrightarrow\,\,
  x\in\Argmin\brk!{f-\scal{\Cdot}{y}},
\end{equation}
where $\Argmin g$ denotes the set of minimizers of a function
$g\colon\HH\to\RX$, that is,
\begin{equation}
  \Argmin g=
  \begin{cases}
    \menge{x\in\HH}{g(x)=\inf g(\HH)}&\text{if}\,\,\dom g\neq\emp,\\
    \emp&\text{otherwise}.
  \end{cases}
\end{equation}
In this section, we derive formulae for evaluating -- and not
merely characterizing -- the conjugate and
subdifferential of a spectral function in terms of those of the associated
invariant function.
Towards this end, we first establish in \cref{t:2} a
reduced minimization principle, which will play a fundamental
role in our analysis in this section as well as in \cref{sec:5}.
In essence, given a spectral function
$\Phi=\varphi\circ\gamma\colon\FH\to\RX$, it describes how the
solutions to
\begin{equation}
\label{e:pbig}
\minimize{X\in\FH}{\Phi(X)-\scal{X}{Y}}
\end{equation}
can be explicitly constructed from solutions
of the reduced problem
\begin{equation}
\label{e:psmal}
\minimize{x\in\XX}{\varphi(x)-\scal{x}{\gamma(Y)}}
\end{equation}
by leveraging the embedding operators
$(\Lambda_{\ad})_{\ad\in\AD}$ and the spectral decomposition
property \cref{d:2b} in \cref{d:2}.

\begin{theorem}[reduced minimization principle]
  \label{t:2}
  Suppose that \cref{a:1} is in force,
  let $\varphi\colon\XX\to\RX$ be proper and $\SD$-invariant,
  and let $Y\in\FH$. Set
  \begin{equation}
    \SX=\Argmin\brk!{\varphi\circ\gamma-\scal{\Cdot}{Y}}
    \quad\text{and}\quad
    S=\Argmin\brk!{\varphi-\scal{\Cdot}{\gamma(Y)}}.
  \end{equation}
  Then
  \begin{equation}
    \label{e:owc1}
    \SX=\menge{\Lambda_{\bd}x}{x\in S\,\,\text{and}\,\,\bd\in\AD_Y}
  \end{equation}
  and
  \begin{equation}
    \label{e:owc2}
    (\forall X\in\FH)\quad
    X\in\SX
    \quad\Leftrightarrow\quad
    \begin{cases}
      \gamma(X)\in S\\
      (\exi\ad\in\AD)\,\,
      X=\Lambda_{\ad}\gamma(X)\,\,\text{and}\,\,
      Y=\Lambda_{\ad}\gamma(Y).
    \end{cases}
  \end{equation}
  Moreover,
  \begin{equation}
  \label{e:owc3}
  \inf_{X\in\FH}\brk!{\varphi\brk!{\gamma(X)}-\scal{X}{Y}}
  =\inf_{x\in\XX}\brk!{\varphi(x)-\scal{x}{\gamma(Y)}}.
  \end{equation}
\end{theorem}
\begin{proof}
Since $\varphi$ is proper,
\cref{p:6}\,\cref{p:6i} ensures that $\varphi\circ\gamma$ is
likewise. Now set
\begin{equation}
\label{e:rlby}
\mu=\inf_{X\in\FH}\brk!{\varphi\brk!{\gamma(X)}-\scal{X}{Y}}
\quad\text{and}\quad
\nu=\inf_{x\in\XX}\brk!{\varphi(x)-\scal{x}{\gamma(Y)}}.
\end{equation}
Using the $\SD$-invariance of $\varphi$ and
\cref{p:6}\,\cref{p:6i}, we get
\begin{align}
(\forall x\in\XX)(\forall\bd\in\AD_Y)\quad
\varphi(x)-\scal{x}{\gamma(Y)}
&=\varphi\brk!{\gamma(\Lambda_{\bd}x)}
-\scal{\Lambda_{\bd}x}{\Lambda_{\bd}\gamma(Y)}
\nonumber\\
&=\varphi\brk!{\gamma(\Lambda_{\bd}x)}-\scal{\Lambda_{\bd}x}{Y}
\label{e:zmc1}\\
&\geq\mu.
\label{e:zmc2}
\end{align}
Hence $\nu\geq\mu$.
On the other hand, it results from property~\cref{d:2c} in
\cref{d:2} that
\begin{equation}
\label{e:5oag}
(\forall X\in\FH)\quad
\varphi\brk!{\gamma(X)}-\scal{X}{Y}
\geq\varphi\brk!{\gamma(X)}-\scal{\gamma(X)}{\gamma(Y)}
\geq\nu.
\end{equation}
Therefore, taking the infimum over $X\in\FH$ yields
\cref{e:owc3}, that is,
\begin{equation}
\label{e:8hif}
\mu=\nu.
\end{equation}
To proceed further, assume that $X\in\SX$. Then
$\varphi(\gamma(X))\in\RR$.
In turn, since \cref{e:5oag} entail that
\begin{equation}
\nu
=\mu
=\varphi\brk!{\gamma(X)}-\scal{X}{Y}
\geq\varphi\brk!{\gamma(X)}-\scal{\gamma(X)}{\gamma(Y)}
\geq\nu,
\end{equation}
we get
\begin{equation}
\label{e:t0u8}
\varphi\brk!{\gamma(X)}-\scal{\gamma(X)}{\gamma(Y)}
=\nu
=\inf_{x\in\XX}\brk!{\varphi(x)-\scal{x}{\gamma(Y)}}
\end{equation}
and
\begin{equation}
\label{e:t0u9}
\scal{X}{Y}=\scal{\gamma(X)}{\gamma(Y)}.
\end{equation}
Hence $\gamma(X)\in S$ and, thanks to \cref{p:3}
(applied to $\DD=\set{X,Y}$),
$\AD_X\cap\AD_Y\neq\emp$.
To summarize, we have established the
inclusion ``$\subset$'' in \cref{e:owc1}, together with the
implication ``$\Rightarrow$'' in \cref{e:owc2}.
To complete the proof, it suffices to establish
the reverse inclusion in \cref{e:owc1}.
In fact, for every $x\in S$ and every $\bd\in\AD_Y$,
we derive from \cref{e:zmc1} and \cref{e:8hif} that
\begin{equation}
\varphi\brk!{\gamma(\Lambda_{\bd}x)}-\scal{\Lambda_{\bd}x}{Y}
=\varphi(x)-\scal{x}{\gamma(Y)}
=\nu
=\mu
\end{equation}
and, therefore, that $\Lambda_{\bd}x\in\SX$.
\end{proof}

\begin{remark}
\label{r:8}
Let us relate \cref{t:2} to the literature.
\begin{enumerate}
\item
\label{r:8i}
The formula \cref{e:owc1} provides a constructive reduction of
\cref{e:pbig} to the reduced problem \cref{e:psmal} as follows:
One first solves \cref{e:psmal},
then finds a spectral decomposition of $Y$,
and finally lifts a solution of \cref{e:psmal} to a solution
of \cref{e:pbig} via the corresponding embedding operator.
\item
\label{r:8ii}
It can be shown that $(\FH,\XX,\gamma)$ is a Fan--Theobald--von
Neumann system in the sense of \cite[Definition~2.2]{Gowda19}.
A result closely related to \cref{t:2} is
\cite[Theorem 3.7\,(ii)]{Gowda19}.
In our setting, applying \cite[Theorem 3.7\,(ii)]{Gowda19}
yields only the equality of infima in
\cref{e:owc3}, as well as the following \emph{nonconstructive}
characterization of elements of $\SX$
(cf.~\cref{e:owc2}):
\begin{equation}
\label{e:ftvn5}
(\forall X\in\FH)\quad
X\in\SX
\quad\Leftrightarrow\quad
\brk[s]!{\,
\gamma(X)\in S\,\,\text{and}\,\,
\scal{X}{Y}=\scal{\gamma(X)}{\gamma(Y)}
\,
}.
\end{equation}
\item
In the eigenvalue decomposition setting of \cref{ex:4}, \cref{t:2}
subsumes, in particular, \cite[Theorem~1]{Benfenati20},
where only the inclusion ``$\supset$'' in \cref{e:owc1} is
established.
\end{enumerate}
\end{remark}

The reduced minimization principle (\cref{t:2}) yields at once the
following calculus rule for evaluating the conjugate of a spectral
function via that of its associated invariant function.

\begin{proposition}
\label{p:8}
Suppose that \cref{a:1} is in force
and let $\varphi\colon\XX\to\RXX$ be $\SD$-invariant.
Then $(\varphi\circ\gamma)^*=\varphi^*\circ\gamma$.
\end{proposition}
\begin{proof}
If $\varphi$ is proper, the desired conjugacy formula follows from
\cref{e:owc3} in \cref{t:2}.
Thus, it remains to address the case where $\varphi$
is not proper.
Suppose that $\varphi$ is not proper.
If $\varphi\equiv\pinf$,
then $\varphi\circ\gamma\equiv\pinf$,
from which and \cref{e:lege} we obtain
$(\varphi\circ\gamma)^*\equiv\minf\equiv\varphi^*\circ\gamma$.
If $\minf\in\varphi(\XX)$,
we deduce from \cref{p:6}\,\cref{p:6i} that
$\minf\in(\varphi\circ\gamma)(\FH)$ and, therefore, from
\cref{e:lege} that
$(\varphi\circ\gamma)^*\equiv\pinf\equiv\varphi^*\circ\gamma$.
\end{proof}

Next, we utilize the above conjugacy formula together with the
Fenchel--Moreau biconjugation theorem to characterize the class of
proper lower semicontinuous convex spectral functions. This task
requires the following invariance property, which is essential
for the remainder of the paper.

\begin{corollary}
\label{c:2}
Suppose that \cref{a:1} is in force
and let $\varphi\colon\XX\to\RXX$ be $\SD$-invariant.
Then $\varphi^*$ is $\SD$-invariant.
\end{corollary}
\begin{proof}
Using \cref{p:6}\,\cref{p:6i} and
\cref{p:8} (applied to the spectral decomposition system
constructed in \cref{ex:9}), we obtain
$\varphi^*=(\varphi\circ\tau)^*=\varphi^*\circ\tau$.
Therefore, invoking \cref{p:6}\,\cref{p:6i} once more, we conclude
that $\varphi^*$ is $\SD$-invariant.
\end{proof}

\begin{corollary}
\label{c:8}
Suppose that \cref{a:1} is in force and let
$\varphi\colon\XX\to\RX$ be proper and $\SD$-invariant.
Then $\varphi\circ\gamma\in\Gamma_0(\FH)$ if and only if
$\varphi\in\Gamma_0(\XX)$.
\end{corollary}
\begin{proof}
\cref{c:2} implies that $\varphi^*$ and $\varphi^{**}$ are
$\SD$-invariant. In addition, since $\varphi$ is proper, we infer
that $\minf\notin\varphi^*(\XX)$
\cite[Proposition~13.10\,(ii)]{Livre1}
and that $\varphi\circ\gamma$ is also proper.
Hence, we derive from the Fenchel--Moreau biconjugation theorem
\cite[Theorem~13.37]{Livre1},
\cref{p:8}, and \cref{p:6}\,\cref{p:6iii} that
$\varphi\circ\gamma\in\Gamma_0(\FH)$
\,$\Leftrightarrow$\,
$(\varphi\circ\gamma)^{**}=\varphi\circ\gamma$
\,$\Leftrightarrow$\,
$\varphi^{**}\circ\gamma=\varphi\circ\gamma$
\,$\Leftrightarrow$\,
$\restr{\varphi^{**}}{\ran\gamma}=
\restr{\varphi}{\ran\gamma}$
\,$\Leftrightarrow$\,
$\varphi^{**}=\varphi$
\,$\Leftrightarrow$\,
$\varphi\in\Gamma_0(\XX)$,
as claimed.
\end{proof}

We now turn to the subdifferentiability of spectral functions.
More precisely, we show how subgradients of a spectral function
can be explicitly constructed -- and not merely characterized --
from those of the associated invariant function.

\begin{proposition}
\label{p:9}
Suppose that \cref{a:1} is in force,
let $\varphi\colon\XX\to\RX$ be proper and $\SD$-invariant,
and let $X\in\FH$. Then
\begin{equation}
\label{e:vep8}
\partial(\varphi\circ\gamma)(X)
=\menge{\Lambda_{\ad}y}{
y\in\partial\varphi\brk!{\gamma(X)}\,\,
\text{and}\,\,\ad\in\AD_X}
\end{equation}
and
\begin{equation}
\label{e:vep9}
(\forall Y\in\FH)\quad
Y\in\partial(\varphi\circ\gamma)(X)
\quad\Leftrightarrow\quad
\begin{cases}
\gamma(Y)\in\partial\varphi\brk!{\gamma(X)}\\
(\exi\ad\in\AD)\,\,
X=\Lambda_{\ad}\gamma(X)\,\,\text{and}\,\,
Y=\Lambda_{\ad}\gamma(Y).
\end{cases}
\end{equation}
\end{proposition}
\begin{proof}
We deduce from \cref{e:vpdz} and \cref{t:2} that
\begin{align}
(\forall Y\in\FH)\quad
Y\in\partial(\varphi\circ\gamma)(X)
&\Leftrightarrow
X\in\Argmin\brk!{\varphi\circ\gamma-\scal{\Cdot}{Y}}
\nonumber\\
&\Leftrightarrow
\begin{cases}
\gamma(X)\in\Argmin\brk!{\varphi-\scal{\Cdot}{\gamma(Y)}}\\
(\exi\ad\in\AD)\,\,
X=\Lambda_{\ad}\gamma(X)\,\,\text{and}\,\,
Y=\Lambda_{\ad}\gamma(Y)
\end{cases}\
\nonumber\\
&\Leftrightarrow
\begin{cases}
\gamma(Y)\in\partial\varphi\brk!{\gamma(X)}\\
(\exi\ad\in\AD)\,\,
X=\Lambda_{\ad}\gamma(X)\,\,\text{and}\,\,
Y=\Lambda_{\ad}\gamma(Y),
\end{cases}
\end{align}
which establishes \cref{e:vep9} and the inclusion ``$\subset$'' in
\cref{e:vep8}. To establish the reverse inclusion in \cref{e:vep8},
suppose that $y\in\partial\varphi(\gamma(X))$ and let
$\ad\in\AD_X$. \cref{c:2} asserts that $\varphi^*$ is
$\SD$-invariant and, in turn, \cref{p:6}\,\cref{p:6i} yields
$\varphi^*\circ\gamma\circ\Lambda_{\ad}=\varphi^*$. Therefore,
by \cref{p:8} and \cref{e:vpdz},
\begin{align}
  (\varphi\circ\gamma)(X)+(\varphi\circ\gamma)^*\brk{\Lambda_{\ad}y}
  &=\varphi\brk!{\gamma(X)}+\varphi^*\brk!{\gamma(\Lambda_{\ad}y)}
  \nonumber\\
  &=\varphi\brk!{\gamma(X)}+\varphi^*(y)
  \nonumber\\
  &=\scal{\gamma(X)}{y}
  \nonumber\\
  &=\scal{\Lambda_{\ad}\gamma(X)}{\Lambda_{\ad}y}
  \nonumber\\
  &=\scal{X}{\Lambda_{\ad}y}.
\end{align}
Consequently, invoking \cref{e:vpdz} once more,
we obtain $\Lambda_{\ad}y\in\partial(\varphi\circ\gamma)(X)$.
\end{proof}

\begin{remark}
\label{r:5}
Let us make a connection between \cref{p:9} and existing works.
\begin{enumerate}
\item
\label{r:5i}
Following up on the discussion in \cref{r:8}\,\cref{r:8ii}, let us
compare \cref{p:9} to \cite[Theorem~4.2]{Jeong24},
which concerns subdifferentials of spectral functions
in the more abstract Fan--Theobald--von Neumann framework.
Applying the latter to the Fan--Theobald--von Neumann system
$(\FH,\XX,\gamma)$ yields only the \emph{characterization}
\begin{equation}
(\forall Y\in\FH)\quad
Y\in\partial(\varphi\circ\gamma)(X)
\quad\Leftrightarrow\quad
\brk[s]!{\,
\gamma(Y)\in\partial\varphi\brk!{\gamma(X)}
\,\,\text{and}\,\,
\scal{X}{Y}=\scal{\gamma(X)}{\gamma(Y)}
\,};
\end{equation}
cf.~\cref{e:vep8}.
\item
The realization of \cref{p:9} in normal decomposition systems
(\cref{ex:2}) is established in \cite[Theorem~4.5]{Lewis96b}.
\item
Specializing the membership characterization \cref{e:vep9} to the
context of Euclidean Jordan algebras (\cref{ex:3}) yields
\cite[Corollary~31]{Baes07}.
\item
Special cases of \cref{p:9} in the context of eigenvalue and
singular value decompositions (\cref{ex:4,ex:5}) can be found in
\cite{Lewis96a,Lewis95}, respectively.
\end{enumerate}
\end{remark}

An immediate consequence of the explicit construction
in \cref{e:vep8} of \cref{p:9} is the following
relationship between the uniqueness of subgradients
of a spectral function and that of its associated invariant
function, which will be required in subsequent proofs.

\begin{corollary}
\label{c:7}
Suppose that \cref{a:1} is in force,
let $\varphi\colon\XX\to\RX$ be proper and $\SD$-invariant,
and let $X\in\FH$. Then
$\partial(\varphi\circ\gamma)(X)$
is a singleton if and only if
$\partial\varphi(\gamma(X))$ is a singleton.
\end{corollary}
\begin{proof}
Recall that a convex subset $C$ of a real Hilbert space $\HH$ is a
singleton if and only if $\menge{\norm{x}}{x\in C}$ is a singleton
\cite[Proposition~3.7]{Livre1}.
It results from \cref{p:9} that
\begin{equation}
\menge{\norm{Y}}{Y\in\partial(\varphi\circ\gamma)(X)}
=\menge{\norm{y}}{y\in\partial\varphi\brk!{\gamma(X)}}.
\end{equation}
Hence, since $\partial(\varphi\circ\gamma)(X)$ and
$\partial\varphi(\gamma(X))$ are both convex
\cite[Proposition~16.4\,(iii)]{Livre1},
the claim follows.
\end{proof}

Next, we leverage the lifting formula for subdifferentials in
\cref{p:9} to study differentiability properties of convex spectral
functions, with particular attention to the distinction between
Fr\'echet and Gateaux differentiability in the infinite-dimensional
setting. Our proof relies on the following well-known
characterizations of these notions. Here $\weakly$ denotes weak
convergence.

\begin{lemma}
\label{l:3}
Let $\HH$ be a real Hilbert space, let $f\in\Gamma_0(\HH)$,
and let $x\in\dom f$.
Then the following hold:
\begin{enumerate}
\item
\label{l:3i}
The following are equivalent:
\begin{enumerate}
\item
\label{l:3ia}
$f$ is Gateaux differentiable at $x$.
\item
\label{l:3ib}
$x\in\intdom f$ and $\partial f(x)$ is a singleton.
\item
\label{l:3ic}
$x\in\intdom f$ and,
for every sequence $(x_n,y_n)_{n\in\NN}$ in $\gra\partial f$ and
every $y\in\partial f(x)$,
$x_n\to x$ \,$\Rightarrow$\, $y_n\weakly y$.
\end{enumerate}
\item
\label{l:3ii}
The following are equivalent:
\begin{enumerate}
\item
\label{l:3iia}
$f$ is Fr\'echet differentiable at $x$.
\item
\label{l:3iib}
$x\in\intdom f$ and,
for every sequence $(x_n,y_n)_{n\in\NN}$ in $\gra\partial f$ and
every $y\in\partial f(x)$,
$x_n\to x$ \,$\Rightarrow$\, $y_n\to y$.
\end{enumerate}
\end{enumerate}
\end{lemma}
\begin{proof}
\cref{l:3ia}\,$\Rightarrow$\,\cref{l:3ib}:
This follows from
\cite[Propositions~17.50 and 17.31\,(i)]{Livre1}.

\cref{l:3ib}\,$\Rightarrow$\,\cref{l:3ia}:
Since a function in $\Gamma_0(\HH)$ is continuous on the interior
of its domain \cite[Proposition~16.27]{Livre1},
$f$ is continuous at $x$.
Therefore, since $\partial f(x)$ is a singleton, we deduce from
\cite[Proposition~17.31\,(ii)]{Livre1} that $f$ is Gateaux
differentiable at $x$.

\cref{l:3ia}\,$\Leftrightarrow$\,\cref{l:3ic}:
Use the equivalence
\cref{l:3ia}\,$\Leftrightarrow$\,\cref{l:3ib}
and \cite[Proposition~17.39]{Livre1}.

\cref{l:3ii}:
See \cite[Proposition~17.41]{Livre1}.
\end{proof}

\begin{proposition}
\label{p:32}
Suppose that \cref{a:1} is in force,
let $\varphi\in\Gamma_0(\XX)$ be $\SD$-invariant,
and let $X\in\FH$ be such that $\gamma(X)\in\dom\varphi$. Then
$\varphi\circ\gamma$ is
Gateaux (resp.~Fr\'echet) differentiable at $X$ if and only if
$\varphi$ is Gateaux (resp.~Fr\'echet) differentiable at
$\gamma(X)$, in which case
\begin{equation}
\label{e:h5k3}
\gamma\brk!{\nabla(\varphi\circ\gamma)(X)}
=\nabla\varphi\brk!{\gamma(X)}
\quad\text{and}\quad
(\forall\ad\in\AD_X)\,\,
\nabla(\varphi\circ\gamma)(X)
=\Lambda_{\ad}\brk!{\nabla\varphi\brk!{\gamma(X)}}.
\end{equation}
\end{proposition}
\begin{proof}
\cref{c:8} ensures that $\varphi\circ\gamma\in\Gamma_0(\FH)$,
while \cref{p:6}\,\cref{p:6iia} gives
$\intdom(\varphi\circ\gamma)=\gamma^{-1}(\intdom\varphi)$.
We derive from \cref{l:3}\,\cref{l:3i} and \cref{c:7} that
\begin{align}
\label{e:9pxp}
\text{$\varphi\circ\gamma$ is Gateaux differentiable at $X$}
&\Leftrightarrow
\text{$X\in\intdom(\varphi\circ\gamma)$
and $\partial(\varphi\circ\gamma)(X)$ is a singleton}
\nonumber\\
&\Leftrightarrow
\text{$\gamma(X)\in\intdom\varphi$ and
$\partial\varphi\brk!{\gamma(X)}$ and is a singleton}
\nonumber\\
&\Leftrightarrow
\text{$\varphi$ is Gateaux differentiable at $\gamma(X)$}.
\end{align}
Next, we establish the equivalence for the case of Fr\'echet
differentiability.
Assume that $\varphi\circ\gamma$ is Fr\'echet differentiable at
$X$ and take $\ad\in\AD_X$.
On the one hand, \cref{p:6}\,\cref{p:6i} yields
$\varphi=(\varphi\circ\gamma)\circ\Lambda_{\ad}$.
On the other hand,
$\Lambda_{\ad}$ is Fr\'echet differentiable at $\gamma(X)$ and
$\varphi\circ\gamma$ is Fr\'echet differentiable at
$X=\Lambda_{\ad}\gamma(X)$.
Hence, the chain rule implies that
$\varphi=(\varphi\circ\gamma)\circ\Lambda_{\ad}$ is Fr\'echet
differentiable at $\gamma(X)$.
Conversely, assume that $\varphi$ is Fr\'echet differentiable at
$\gamma(X)$.
It results from \cref{l:3}\,\cref{l:3ii} that
$\gamma(X)\in\intdom\varphi$ and, therefore, that
$X\in\gamma^{-1}(\intdom\varphi)=\intdom(\varphi\circ\gamma)$.
Now let $Y\in\partial(\varphi\circ\gamma)(X)$
and let $(X_n,Y_n)_{n\in\NN}$ be a sequence in
$\gra\partial(\varphi\circ\gamma)$ such that
$X_n\to X$. By \cref{p:9},
$\gamma(Y)\in\partial\varphi\brk{\gamma(X)}$
and $(\forall n\in\NN)$
$\gamma(Y_n)\in\partial\varphi\brk{\gamma(X_n)}$.
At the same time, the nonexpansiveness of $\gamma$
(\cref{p:2}\,\cref{p:2iv}) gives
$\norm{\gamma(X_n)-\gamma(X)}\leq\norm{X_n-X}\to 0$.
Thus, because $\varphi$ is Fr\'echet differentiable at
$\gamma(X)$, it follows from
\cref{l:3}\,\cref{l:3ii} that $\gamma(Y_n)\to\gamma(Y)$
and, in turn, from \cref{p:2}\,\cref{p:2iii} that
\begin{equation}
\label{e:tng1}
\lim_{n\to\pinf}\norm{Y_n}
=\lim_{n\to\pinf}\norm{\gamma(Y_n)}
=\norm{\gamma(Y)}
=\norm{Y}.
\end{equation}
Additionally, since \cref{e:9pxp} implies that $\varphi\circ\gamma$ is
Gateaux differentiable at $X$, the equivalence
\cref{l:3ia}\,$\Leftrightarrow$\,\cref{l:3ic} in
\cref{l:3}\,\cref{l:3i} entails that
$Y_n\weakly Y$. This and \cref{e:tng1} force
$Y_n\to Y$ \cite[Lemma~2.51\,(i)]{Livre1}.
Consequently, invoking \cref{l:3}\,\cref{l:3ii} once more,
we deduce that $\varphi\circ\gamma$ is Fr\'echet differentiable at
$X$. Finally, in both cases, \cref{e:h5k3} follows from \cref{p:9}
and \cite[Proposition~17.31\,(i)]{Livre1}.
\end{proof}

We end this section with a complement to the reduced minimization
principle of \cref{t:2}. In \cref{p:54}, we provide conditions under
which the uniqueness property of solutions to \cref{e:pbig} is
equivalent to that of the reduced problem \cref{e:psmal}.
Doing so requires the following generalization of Ky Fan's
majorization theorem, which compares the spectrum of a sum and the
sum of the spectra.

\begin{proposition}
\label{p:4}
Suppose that \cref{a:1} is in force and that, in addition,
$\XX$ is finite-dimensional.
Let $m\in\NN\smallsetminus\set{0}$,
let $(X_i)_{1\leq i\leq m}\in\FH^m$,
and let $(\alpha_i)_{1\leq i\leq m}\in\RP^m$.
Then
\begin{equation}
\gamma\brk3{\sum_{i=1}^m\alpha_iX_i}
\in\conv\brk4{\SD\cdot\sum_{i=1}^m\alpha_i\gamma(X_i)},
\end{equation}
in which we employ the notation $(\forall x\in\XX)$
$\SD\cdot x=\menge{\sd\cdot x}{\sd\in\SD}$.
\end{proposition}
\begin{proof}
Set
\begin{equation}
X=\sum_{i=1}^m\alpha_iX_i,
\quad
y=\sum_{i=1}^m\alpha_i\gamma(X_i),
\quad\text{and}\quad
C=\conv(\SD\cdot y).
\end{equation}
First, we claim that
\begin{equation}
\label{e:s56u}
\text{$C$ is compact.}
\end{equation}
By virtue of \cite[Corollary~2.30]{Rock09},
it suffices to show that $\SD\cdot y$ is compact.
Observe that $\SD\cdot y$ is bounded thanks to $(\forall\sd\in\SD)$
$\norm{\sd\cdot y}=\norm{y}$.
Now let $(\sd_n)_{n\in\NN}$ be a sequence in $\SD$ such that the
sequence $(\sd_n\cdot y)_{n\in\NN}$ converges to some $z\in\XX$.
Since $\ran\gamma=\ran\tau$ is a convex cone
(\cref{p:2}\,\cref{p:2ii}), we get $y\in\ran\tau$. Hence, by
\cref{p:1}\,\cref{p:1i},
\begin{equation}
\label{e:qphs}
y=\tau(y).
\end{equation}
In turn, using the $\SD$-invariance and nonexpansiveness
(\cref{p:1}\,\cref{p:1iv}) of $\tau$, we obtain
$\norm{y-\tau(z)}
=\norm{\tau(y)-\tau(z)}
=\norm{\tau(\sd_n\cdot y)-\tau(z)}
\leq\norm{\sd_n\cdot y-z}
\to 0$.
Hence $\tau(z)=y$, from which and \cref{e:1ae1} we get
$z\in\SD\cdot y$.
This shows that $\SD\cdot y$ is closed.
Thus, inasmuch as $\XX$ is finite-dimensional,
$\SD\cdot y$ is compact, as claimed.
Next, we derive from property~\cref{d:2c} in \cref{d:2},
\cref{e:1ae2}, and \cref{e:qphs} that
\begin{align}
(\forall z\in\XX)(\forall\ad\in\AD_X)\quad
\scal{\gamma(X)}{z}
&=\scal{\Lambda_{\ad}\gamma(X)}{\Lambda_{\ad}z}
\nonumber\\
&=\scal{X}{\Lambda_{\ad}z}
\nonumber\\
&=\sum_{i=1}^m\alpha_i\scal{X_i}{\Lambda_{\ad}z}
\nonumber\\
&\leq\sum_{i=1}^m\alpha_i
\scal!{\gamma(X_i)}{\gamma\brk{\Lambda_{\ad}z}}
\nonumber\\
&=\scal{y}{\tau(z)}.
\end{align}
On the other hand, for every $z\in\XX$,
upon choosing $\ts\in\SD$ for which
$z=\ts\cdot\tau(z)$, we deduce from
\cref{l:1-}\,\cref{l:1-ii} and
\cite[Proposition~11.1\,(iii)]{Livre1} that
\begin{equation}
\scal{y}{\tau(z)}
=\scal{\ts\cdot y}{\ts\cdot\tau(z)}
=\scal{\ts\cdot y}{z}
\leq\sup_{\sd\in\SD}\scal{\sd\cdot y}{z}
=\sup_{w\in C}\scal{w}{z}.
\end{equation}
To summarize, we have shown that
$(\forall z\in\XX)$
$\scal{\gamma(X)}{z}\leq\sup_{w\in C}\scal{w}{z}$.
Consequently, since $C$ is closed and convex, we conclude via a
consequence of the separation theorem
\cite[Proposition~7.11]{Livre1} that
$\gamma(X)\in C$.
\end{proof}

\begin{remark}
\label{r:1}
Here are some noteworthy instances of \cref{p:4}.
\begin{enumerate}
\item
Specializing \cref{p:4} to the Euclidean Jordan algebra setting of
\cref{ex:3} yields
\begin{equation}
(\forall X_1\in\FH)
(\forall X_2\in\FH)\quad
\lambda(X_1+X_2)\in\conv\brk2{\PS^N\cdot\brk!{\lambda(X_1)+\lambda(X_2)}}.
\end{equation}
In particular, by choosing $\FH=\HS^N(\KB)$ (see \cref{ex:4}),
we obtain
\begin{equation}
\brk!{\forall X_1\in\HS^N(\KB)}
\brk!{\forall X_2\in\HS^N(\KB)}\quad
\lambda(X_1+X_2)\in\conv\brk2{\PS^N\cdot\brk!{\lambda(X_1)+\lambda(X_2)}},
\end{equation}
which is Ky Fan's majorization theorem for the eigenvalues of a
sum of Hermitian matrices in the real and complex cases
\cite[Theorem~9.G.1]{Marshall11}.
\item
In the singular value decomposition setting of \cref{ex:5},
we deduce from \cref{p:4} the inclusion
\begin{equation}
\brk!{\forall X_1\in\CC^{M\times N}}
\brk!{\forall X_2\in\CC^{M\times N}}\quad
\sigma(X_1+X_2)\in\conv\brk2{\PS_{\pm}^m\cdot
\brk!{\sigma(X_1)+\sigma(X_2)}},
\end{equation}
which is Ky Fan's weak majorization theorem for the singular
values of a sum of rectangular matrices \cite[9.G.1.d]{Marshall11}.
\item
We infer from \cite[Propositions~3.3 and 3.4]{Lewis00} that the
semisimple real Lie algebra setting of \cite[Theorem~2]{Tam98} is a
normal decomposition system (see \cref{ex:2}),
and we therefore recover from \cref{p:4} the extension
\cite[Theorem~2]{Tam98} of the above two Ky Fan's majorization
theorems to semisimple real Lie algebras.
\end{enumerate}
\end{remark}

\begin{proposition}
\label{p:54}
Consider the setting of \cref{t:2}. Then the following hold:
\begin{enumerate}
\item
\label{p:54i}
Suppose that $\XX$ is finite-dimensional. Then
$\SX$ is convex if and only if $S$ is convex.
\item
\label{p:54ii}
Suppose that one of the following is satisfied:
\begin{enumerate}
\item
\label{p:54iia}
$\XX$ is finite-dimensional.
\item
\label{p:54iib}
$\varphi\in\Gamma_0(\XX)$.
\end{enumerate}
Then $\SX$ is a singleton if and only if $S$ is a singleton.
\end{enumerate}
\end{proposition}
\begin{proof}
\cref{p:54i}:
Take $\alpha\in\zeroun$ and $\bd\in\AD_Y$.
Suppose first that $\SX$ is convex and that $x_0$
and $x_1$ lie in $S$, and set $x=(1-\alpha)x_0+\alpha x_1$.
On account of \cref{t:2}, the points $\Lambda_{\bd}x_0$ and
$\Lambda_{\bd}x_1$ lie in the convex set $\SX$, which leads to
$\Lambda_{\bd}x=(1-\alpha)\Lambda_{\bd}x_0+\alpha\Lambda_{\bd}x_1\in\SX$.
Hence, we derive from \cref{e:owc3} in \cref{t:2} and
\cref{p:6}\,\cref{p:6i} that
\begin{align}
\inf_{x\in\XX}\brk!{\varphi(x)-\scal{x}{\gamma(Y)}}
&=\inf_{X\in\FH}\brk!{\varphi\brk!{\gamma(X)}-\scal{X}{Y}}
\nonumber\\
&=\varphi\brk!{\gamma(\Lambda_{\bd}x)}-\scal{\Lambda_{\bd}x}{Y}
\nonumber\\
&=\varphi(x)-\scal{\Lambda_{\bd}x}{\Lambda_{\bd}\gamma(Y)}
\nonumber\\
&=\varphi(x)-\scal{x}{\gamma(Y)},
\end{align}
from which it follows that $x\in S$.
Conversely, suppose that $S$ is convex and that $X_0$ and $X_1$ lie in
$\SX$, and set $X=(1-\alpha)X_0+\alpha X_1$.
In the light of \cref{t:2}, we must show that
$\AD_X\cap\AD_Y\neq\emp$ and that
$\gamma(X)\in S$.
We infer from \cref{t:2} and \cref{p:3} that
\begin{equation}
\label{e:8gma}
\brk!{\forall i\in\set{0,1}}\quad
\gamma(X_i)\in S
\quad\text{and}\quad
\scal{X_i}{Y}=\scal{\gamma(X_i)}{\gamma(Y)}.
\end{equation}
At the same time, \cref{p:4} ensures the existence of finite families
$(\sd_j)_{j\in J}$ in $\SD$ and $(\beta_j)_{j\in J}$ in $\rzeroun$ such
that
\begin{equation}
\label{e:uvap}
\sum_{j\in J}\beta_j=1
\quad\text{and}\quad
\gamma(X)=\sum_{j\in J}\beta_j\brk2{
\sd_j\cdot\brk!{(1-\alpha)\gamma(X_0)+\alpha\gamma(X_1)}}.
\end{equation}
In turn, it results from \cref{p:1}\,\cref{p:1iii},
the $\SD$-invariance of $\tau$,
\cref{p:2}\,\cref{p:2i}, and
\cref{e:8gma} that
\begin{equation}
\label{e:l0bf}
\brk!{\forall i\in\set{0,1}}
(\forall j\in J)\quad
\scal{\sd_j\cdot\gamma(X_i)}{\gamma(Y)}
\leq\scal{\gamma(X_i)}{\gamma(Y)}
=\scal{X_i}{Y}.
\end{equation}
Combining this with property~\cref{d:2c} in \cref{d:2} and \cref{e:uvap}, we obtain
\begin{align}
\scal{X}{Y}
&\leq\scal{\gamma(X)}{\gamma(Y)}
\nonumber\\
&=\sum_{j\in J}\beta_j\brk2{
(1-\alpha)\scal{\sd_j\cdot\gamma(X_0)}{\gamma(Y)}
+\alpha\scal{\sd_j\cdot\gamma(X_1)}{\gamma(Y)}}
\nonumber\\
&\leq\sum_{j\in J}\beta_j\brk2{
(1-\alpha)\scal{X_0}{Y}+\alpha\scal{X_1}{Y}}
\nonumber\\
&=\scal{X}{Y}.
\end{align}
Hence
\begin{equation}
\label{e:03ru}
\scal{X}{Y}=\scal{\gamma(X)}{\gamma(Y)}
\end{equation}
and, thanks to \cref{e:l0bf},
\begin{equation}
\label{e:badv}
\brk!{\forall i\in\set{0,1}}
(\forall j\in J)\quad
\scal{\sd_j\cdot\gamma(X_i)}{\gamma(Y)}
=\scal{\gamma(X_i)}{\gamma(Y)}.
\end{equation}
It results from \cref{e:03ru} and \cref{p:3} that
$\AD_X\cap\AD_Y\neq\emp$, and
it thus remains to show that $\gamma(X)\in S$.
For every $i\in\set{0,1}$ and every $j\in J$,
using the $\SD$-invariance of $\varphi$,
\cref{e:badv}, and \cref{e:8gma},
we get
\begin{equation}
\varphi\brk!{\sd_j\cdot\gamma(X_i)}-\scal{\sd_j\cdot\gamma(X_i)}{\gamma(Y)}
=\varphi\brk!{\gamma(X_i)}-\scal{\gamma(X_i)}{\gamma(Y)}
=\min_{z\in\XX}\brk!{\varphi(z)-\scal{z}{\gamma(Y)}},
\end{equation}
which verifies that $\sd_j\cdot\gamma(X_i)\in S$.
Consequently, appealing to \cref{e:uvap}
and the convexity of $S$, we conclude that $\gamma(X)\in S$.

\cref{p:54ii}:
Recall that a convex subset $C$ of a real Hilbert space $\HH$ is a
singleton if and only if $\menge{\norm{x}}{x\in C}$ is a singleton
\cite[Proposition~3.7]{Livre1}.
Additionally, thanks to \cref{t:2},
$\menge{\norm{X}}{X\in\SX}=\menge{\norm{x}}{x\in S}$.

\cref{p:54iia}:
We derive from \cref{p:54i} that
\begin{align}
\text{$\SX$ is a singleton}
&\Leftrightarrow
\text{$\SX$ is convex and $\menge{\norm{X}}{X\in\SX}$
is a singleton}
\nonumber\\
&\Leftrightarrow
\text{$S$ is convex and $\menge{\norm{x}}{x\in S}$
is a singleton}
\nonumber\\
&\Leftrightarrow
\text{$S$ is a singleton.}
\end{align}

\cref{p:54iib}:
\cref{c:8} ensures that $\varphi\circ\gamma\in\Gamma_0(\FH)$.
Hence $\SX$ and $S$ are both convex as the sets of minimizers
of convex functions. Therefore,
\begin{align}
\text{$\SX$ is a singleton}
&\Leftrightarrow
\text{$\menge{\norm{X}}{X\in\SX}$ is a singleton}
\nonumber\\
&\Leftrightarrow
\text{$\menge{\norm{x}}{x\in S}$ is a singleton}
\nonumber\\
&\Leftrightarrow
\text{$S$ is a singleton,}
\end{align}
which completes the proof.
\end{proof}

\section{Bregman proximity operators of spectral functions}
\label{sec:5}

The classical proximity (or proximal point) operators
\cite{Moreau62} are the basic building blocks of modern first-order
nonsmooth optimization algorithms
\cite{Livre1,Beck17,ProxRespo,ClasonValkonen,Combettes2024-acnu}.
The notion of a Bregman proximity operator \cite{Bauschke03}
generalizes this fundamental concept beyond Hilbert spaces and
allows more algorithmic and computational flexibility in its
evaluation \cite{Bauschke17,Bauschke03b,Combettes16,Nguyen17}.
In this section, we employ the reduced minimization principle
of \cref{t:2} to devise explicit constructions for
Bregman proximal points of a spectral function by first
reducing their computation to that of the associated invariant
function and then identifying a suitable spectral decomposition;
see \cref{t:4}.
This extends and strengthens existing results in the
eigenvalue and singular value decomposition settings
of \cref{ex:4,ex:5}, and appears to be novel
in the context of \cref{ex:8,ex:7,ex:2,ex:3,ex:6}, even for
the case of classical proximity operators of convex spectral
functions; see \cref{r:9} for a detailed discussion.
More importantly, \cref{t:4} further illustrates the algorithmic
advantages of our framework through the introduction
of the spectral decomposition property \cref{d:2b} in \cref{d:2}.

To state our results, we first recall the notions of Legendre
functions and Bregman proximity operators in infinite-dimensional
spaces \cite{Bauschke01,Bauschke03}.
Let $\HH$ be a real Hilbert space.
Following \cite[Definition~5.2]{Bauschke01} for the case of real
Banach spaces, we say that a function $g\in\Gamma_0(\HH)$ is:
\begin{itemize}
\item
essentially smooth if $\partial g$ is both locally bounded and
single-valued on its domain.
\item
essentially strictly convex if
$(\partial g)^{-1}$ is locally bounded on its domain and
$g$ strictly convex on every convex subset of $\dom\partial g$.
\item
a Legendre function if it is both essentially smooth and essentially
strictly convex.
\end{itemize}
(An operator $M\colon\HH\to 2^{\HH}$
is locally bounded at a point $x\in\HH$
if there exists $\delta\in\RPP$ such that the set
$\bigcup_{z\in B(x;\delta)}Az$ is bounded.)
We shall require the following characterization
of essential smoothness and essential strict convexity from
\cite{Bauschke01}: For every $g\in\Gamma_0(\HH)$,
\begin{equation}
\label{e:esmo}
\text{$g$ is essentially smooth}
\quad\Leftrightarrow\quad
\begin{cases}
\intdom g\neq\emp\\
\text{$\partial g$ is single-valued on its domain}
\end{cases}
\end{equation}
and
\begin{equation}
\label{e:smod}
\text{$g$ is essentially strictly convex}
\quad\Leftrightarrow\quad
\text{$g^*$ is essentially smooth};
\end{equation}
see \cite[Theorems~5.6 and 5.4]{Bauschke01}, respectively.
Now let $g\in\Gamma_0(\HH)$ be a Legendre function.
The Bregman distance associated with $g$ is
\begin{equation}
\label{e:bdist}
\begin{aligned}
D_g\colon\HH\times\HH&\to\RPX\\
(y,x)&\mapsto
\begin{cases}
g(y)-g(x)-\scal{y-x}{\nabla g(x)}&\text{if}\,\,x\in\intdom g,\\
\pinf&\text{otherwise}.
\end{cases}
\end{aligned}
\end{equation}
Given a function $f\colon\HH\to\RX$ such that 
$(\dom f)\cap(\dom g)\neq\emp$,
the Bregman envelope of $f$ with respect to $g$ is
\begin{equation}
\label{e:benv}
\env_f^g\colon\HH\to\RXX\colon
x\mapsto\inf_{y\in\HH}\brk!{f(y)+D_g(y,x)},
\end{equation}
and the Bregman proximity operator
(or, in the terminology of \cite{Bauschke03}, $D$-prox operator)
of $f$ with respect to $g$ is
\begin{equation}
\label{e:bprox}
\begin{aligned}
\Prox_f^g\colon\HH&\to 2^{\HH}\\
x&\mapsto\Argmin\brk!{f+D_g(\Cdot,x)}
=\begin{cases}
\Argmin\brk!{f+g-\scal{\Cdot}{\nabla g(x)}}
&\text{if}\,\,x\in\intdom g,\\
\emp&\text{otherwise}.
\end{cases}
\end{aligned}
\end{equation}
In particular, given a subset $C$ of $\HH$ such that
$C\cap\dom g\neq\emp$, the Bregman distance to $C$ with respect to
$g$ is
\begin{equation}
  \dist_C^g=\env_{\iota_C}^g,
\end{equation}
and the Bregman projector onto $C$ with respect to $g$ is
\begin{equation}
  \Proj_C^g=\Prox_{\iota_C}^g.
\end{equation}
Finally, when $g=\norm{\Cdot}^2/2$, we shall omit the superscript
and simply write $\env_f$, $\Prox_f$, and $\Proj_C$.
(Note that $\dist_C^{\norm{\Cdot}^2/2}=d_C^2$, where
$d_C\colon x\mapsto\inf\norm{x-C}$ is the distance function to $C$,
thereby eliminates the need for simplifying
$\dist_C^{\norm{\Cdot}^2/2}$ to $\dist_C$.)

The evaluation of Bregman proximal points relies fundamentally
on the Legendre property of the auxiliary function. Accordingly, to
facilitate the reduction of such computation for
a spectral function to that of the associated invariant function,
we next characterize the class of spectral Legendre functions.

\begin{proposition}
  \label{p:7}
  Suppose that \cref{a:1} is in force and let $\varphi\in\Gamma_0(\XX)$
  be $\SD$-invariant. Then the following hold:
  \begin{enumerate}
    \item
      \label{p:7i}
      $\varphi\circ\gamma$ is essentially smooth if and only if
      $\varphi$ is essentially smooth.
    \item
      \label{p:7ii}
      $\varphi\circ\gamma$ is essentially strictly convex if and only if
      $\varphi$ is essentially strictly convex.
    \item
      \label{p:7iii}
      $\varphi\circ\gamma$ is a Legendre function if and only if
      $\varphi$ is a Legendre function.
  \end{enumerate}
\end{proposition}
\begin{proof}
\cref{c:8} states that $\varphi\circ\gamma\in\Gamma_0(\FH)$.

\cref{p:7i}:
Suppose that $\varphi\circ\gamma$ is essentially smooth.
Thanks to \cref{e:esmo} and \cref{p:6}\,\cref{p:6iib},
$\intdom\varphi\neq\emp$.
Now take $x\in\dom\partial\varphi$ and $\ad\in\AD$.
By \cref{p:6}\,\cref{p:6i},
$\varphi=\varphi\circ\tau$. Thus,
in the spectral decomposition system
$(\XX,\SD,\tau,(\pi_{\sd})_{\sd\in\SD})$ 
for $\XX$ constructed in \cref{ex:9},
we deduce from \cref{p:9} that
\begin{equation}
\partial\varphi(x)
=\partial(\varphi\circ\tau)(x)
=\menge{\sd\cdot y}{y\in\partial\varphi\brk!{\tau(x)}
\,\,\text{and}\,\,
\sd\in\SD\,\,\text{such that}\,\,
x=\sd\cdot\tau(x)
},
\end{equation}
and from \cref{c:7} that
\begin{equation}
\label{e:ymdl}
\text{$\partial\varphi(x)$ is a singleton}
\quad\Leftrightarrow\quad
\text{$\partial\varphi\brk!{\tau(x)}$ is a singleton}.
\end{equation}
Therefore, because $\partial\varphi(x)$ is nonempty,
$\partial\varphi(\tau(x))=\partial\varphi(\gamma(\Lambda_{\ad}x))$
is likewise, where the equality follows from \cref{e:1ae2}.
This and \cref{p:9} imply that
$\partial(\varphi\circ\gamma)(\Lambda_{\ad}x)\neq\emp$.
Hence, the essential smoothness of
$\varphi\circ\gamma$ forces
$\partial(\varphi\circ\gamma)(\Lambda_{\ad}x)$ to be a singleton
(see \cref{e:esmo}) and, therefore, \cref{c:7} entails that
$\partial\varphi(\gamma(\Lambda_{\ad}x))
=\partial\varphi(\tau(x))$ is likewise.
As a result, it follows from \cref{e:ymdl} that $\partial\varphi(x)$
singleton. Thus, $\varphi$ is essentially smooth.
Conversely, suppose that $\varphi$ is essentially smooth and let
$X\in\dom\partial(\varphi\circ\gamma)$.
Thanks to \cref{e:esmo} and \cref{p:6}\,\cref{p:6iib},
$\intdom(\varphi\circ\gamma)\neq\emp$.
On the other hand, \cref{p:9} implies that
$\gamma(X)\in\dom\partial\varphi$ and, in turn,
\cref{e:esmo} entails that $\partial\varphi(\gamma(X))$ is a
singleton. Therefore, we infer from
\cref{c:7} that $\partial(\varphi\circ\gamma)(X)$ is also a
singleton and conclude via \cref{e:esmo} that
$\varphi\circ\gamma$ is essentially smooth.

\cref{p:7ii}:
We deduce from \cref{c:2} and a consequence of the
Fenchel--Moreau biconjugation theorem
(\cite[Corollary~13.38]{Livre1}) that $\varphi^*$ is an
$\SD$-invariant function in $\Gamma_0(\XX)$.
Thus, on account of \cref{e:smod}, \cref{p:8}, and \cref{p:7i}, we
obtain
\begin{align}
\text{$\varphi\circ\gamma$ is essentially strictly convex}
&\Leftrightarrow
\text{$\varphi^*\circ\gamma=(\varphi\circ\gamma)^*$ is essentially
smooth}
\nonumber\\
&\Leftrightarrow
\text{$\varphi^*$ is essentially smooth}
\nonumber\\
&\Leftrightarrow
\text{$\varphi$ is essentially strictly convex},
\end{align}
as claimed.

\cref{p:7iii}:
Combine \cref{p:7i,p:7ii}.
\end{proof}

We are now in a position to express Bregman envelopes and proximity
operators of a spectral function in terms of those of the
associated invariant function.
We emphasize that item \cref{t:4ii} in \cref{t:4}
provides a complete, \emph{constructive}
description of the set-valued operator
$\Prox_{\varphi\circ\gamma}^{\psi\circ\gamma}$
by lifting $\Prox_{\varphi}^{\psi}$
from the smaller space $\XX$ to $\FH$ via
the spectral decomposition property \cref{d:2b} in \cref{d:2}.
To the best of our knowledge, this result is novel, even in the
well-studied Hermitian matrix setting of \cref{ex:4}.

\begin{theorem}
\label{t:4}
Suppose that \cref{a:1} is in force. Let $\varphi\colon\XX\to\RX$ be a
proper $\SD$-invariant function, let $\psi\in\Gamma_0(\XX)$ be an
$\SD$-invariant Legendre function such that
$(\dom\varphi)\cap(\dom\psi)\neq\emp$, and let $X\in\FH$.
Then the following hold:
\begin{enumerate}
\item
\label{t:4i}
$\env_{\varphi\circ\gamma}^{\psi\circ\gamma}X
=\brk{\env_{\varphi}^{\psi}}(\gamma(X))$.
\item
\label{t:4ii}
$\Prox_{\varphi\circ\gamma}^{\psi\circ\gamma}X
=\menge{\Lambda_{\ad}z}{
z\in\Prox_{\varphi}^{\psi}\gamma(X)
\,\,\text{and}\,\,\ad\in\AD_X}$.
\item
\label{t:4iii}
For every $Z\in\FH$,
\begin{equation}
\label{e:utib}
Z\in\Prox_{\varphi\circ\gamma}^{\psi\circ\gamma}X
\quad\Leftrightarrow\quad
\begin{cases}
\gamma(Z)\in\Prox_{\varphi}^{\psi}\gamma(X)\\
(\exi\ad\in\AD)\,\,
X=\Lambda_{\ad}\gamma(X)\,\,\text{and}\,\,
Z=\Lambda_{\ad}\gamma(Z).
\end{cases}
\end{equation}
\item
\label{t:4iv}
Suppose that one of the following is satisfied:
\begin{enumerate}
\item
\label{t:4iva}
$\XX$ is finite-dimensional.
\item
\label{t:4ivb}
$\varphi\in\Gamma_0(\XX)$.
\end{enumerate}
Then $\Prox_{\varphi\circ\gamma}^{\psi\circ\gamma}X$ is a singleton
if and only if $\Prox_{\varphi}^{\psi}\gamma(X)$ is a singleton.
\end{enumerate}
\end{theorem}
\begin{proof}
\cref{p:7}\,\cref{p:7iii} asserts that $\psi\circ\gamma$ is a Legendre
function, while \cref{p:6}\,\cref{p:6iia} states that
$\intdom(\psi\circ\gamma)=\gamma^{-1}(\intdom\psi)$.
We assume henceforth that
\begin{equation}
\gamma(X)\in\intdom\psi
\end{equation}
since otherwise the assertions are clear from the definition
(see \cref{e:bdist,e:benv,e:bprox}).
Then $\psi$ is Gateaux differentiable at $\gamma(X)$
\cite[Theorem~5.6]{Bauschke01} and \cref{p:32} thus entails that
$\psi\circ\gamma$ is Gateaux differentiable at $X$.
In turn, upon setting
\begin{equation}
\label{e:783g}
Y=\nabla(\psi\circ\gamma)(X)
\quad\text{and}\quad
y=\nabla\psi\brk!{\gamma(X)},
\end{equation}
we get from \cref{p:32} that
\begin{equation}
\label{e:nkgu}
\gamma(Y)=y
\quad\text{and}\quad
(\forall\ad\in\AD_X)\,\,
Y=\Lambda_{\ad}y=\Lambda_{\ad}\gamma(Y).
\end{equation}
Hence, by \cref{p:3},
\begin{equation}
\label{e:0yfe}
\scal{X}{Y}=\scal{\gamma(X)}{\gamma(Y)}=\scal{\gamma(X)}{y}.
\end{equation}
We claim that
\begin{equation}
\label{e:gxgy}
\AD_X=\AD_Y.
\end{equation}
Indeed, it results from \cref{e:nkgu} that $\AD_X\subset\AD_Y$.
Now take $\bd\in\AD_Y$ and set $W=\Lambda_{\bd}\gamma(X)$.
\cref{p:2}\,\cref{p:2i} implies that $\gamma(W)=\gamma(X)$ and,
therefore, that
\begin{equation}
\label{e:qb5z}
\Lambda_{\bd}\gamma(W)
=\Lambda_{\bd}\gamma(X)
=W
\end{equation}
Thus $\bd\in\AD_W$. Additionally, we have
\begin{equation}
\partial\psi\brk!{\gamma(W)}
=\partial\psi\brk!{\gamma(X)}
=\set!{\nabla\psi\brk!{\gamma(X)}},
\end{equation}
where the last equality follows from the Gateaux
differentiability of $\psi$ at $\gamma(X)$.
Hence, on account of \cref{p:9},
$\Lambda_{\bd}\brk{\nabla\psi\brk{\gamma(X)}}
\in\partial(\psi\circ\gamma)(W)$
and, because $\psi\circ\gamma$ is a Legendre function, we must have
$\nabla(\psi\circ\gamma)(W)
=\Lambda_{\bd}\brk{\nabla\psi\brk{\gamma(X)}}$.
In turn, since $\bd\in\AD_Y$, we derive from \cref{e:783g} and
\cref{e:nkgu} that
\begin{equation}
\nabla(\psi\circ\gamma)(W)
=\Lambda_{\bd}\gamma(Y)
=Y
=\nabla(\psi\circ\gamma)(X).
\end{equation}
On the other hand, since $\psi\circ\gamma$ is a Legendre function,
$\nabla(\psi\circ\gamma)\colon\intdom(\psi\circ\gamma)\to
\intdom(\psi\circ\gamma)^*$ is a bijection
\cite[Theorem~5.10]{Bauschke01}.
Thus $X=W=\Lambda_{\bd}\gamma(X)$ and, consequently,
$\bd\in\AD_X$.

\cref{t:4i}:
We derive from \cref{e:bdist}, \cref{e:783g},
\cref{t:2} (applied to the proper $\SD$-invariant function
$\varphi+\psi$),
\cref{e:nkgu}, and \cref{e:0yfe} that
\begin{align}
\env_{\varphi\circ\gamma}^{\psi\circ\gamma}X
&=\inf_{Z\in\FH}\brk!{\varphi\brk!{\gamma(Z)}+D_{\psi\circ\gamma}(Z,X)}
\nonumber\\
&=\inf_{Z\in\FH}\brk!{
\varphi\brk!{\gamma(Z)}+\psi\brk!{\gamma(Z)}
-\psi\brk!{\gamma(X)}-\scal{Z}{Y}+\scal{X}{Y}}
\label{e:mnq1}\\
&=\inf_{z\in\XX}\brk!{\varphi(z)+\psi(z)-\psi\brk!{\gamma(X)}
-\scal{z}{y}+\scal{\gamma(X)}{y}}
\label{e:mnq2}\\
&=\brk!{\env_{\varphi}^{\psi}}\brk!{\gamma(X)},
\end{align}
as desired.

\cref{t:4ii}--\cref{t:4iv}:
On account of \cref{e:bprox} and \cref{e:783g}, we have
\begin{equation}
\label{e:bxy1}
\Prox_{\varphi\circ\gamma}^{\psi\circ\gamma}X
=\Argmin\brk!{\varphi\circ\gamma+\psi\circ\gamma-\scal{\Cdot}{Y}}
\end{equation}
and, by using \cref{e:nkgu},
\begin{equation}
\label{e:bxy2}
\Prox_{\varphi}^{\psi}\gamma(X)
=\Argmin\brk!{\varphi+\psi-\scal{\Cdot}{y}}
=\Argmin\brk!{\varphi+\psi-\scal{\Cdot}{\gamma(Y)}}.
\end{equation}
Hence, in the light of \cref{e:gxgy},
\cref{t:4ii,t:4iii} follow from \cref{t:2},
and \cref{t:4iv} follows from \cref{p:54}\,\cref{p:54ii}
(applied to the function $\varphi+\psi$).
\end{proof}

\begin{remark}
\label{r:9}
Let us relate \cref{t:4} to existing works.
\begin{enumerate}
\item
Consider the special case of \cref{t:4} where
$\varphi\in\Gamma_0(\XX)$ and $\psi=\norm{\Cdot}^2/2$.
\begin{itemize}
\item
In the context of \cref{ex:1}, \cref{t:4}\,\cref{t:4ii} reduces to
\cite[Theorem~6.18]{Beck17}.
\item
In the eigenvalue and singular value decomposition settings of
\cref{ex:4,ex:5}, respectively, where $\KB$ is $\RR$ or $\CC$,
we recover from items
\cref{t:4ii} and \cref{t:4ivb} in \cref{t:4} well-known expressions
for the proximity operator of a lower semicontinuous convex
spectral function; see, for instance,
\cite[Corollary~24.65 and Proposition~24.68]{Livre1} and
\cite[Theorems~7.18 and 7.29]{Beck17}.
\item
In the singular value decomposition framework for matrices
of quaternions, that is, \cref{ex:5} where $\KB=\HB$, concrete
instantiations of \cref{t:4}\,\cref{t:4ii} arising in
machine learning applications can be found in \cite{Chan16,Chen19}.
\item
To the best of our knowledge, \cref{t:4} is new in the context of
\cref{ex:8bis,ex:8,ex:7,ex:2,ex:3,ex:6}.
\end{itemize}
\item
In the eigenvalue decomposition setting of \cref{ex:4} with $\KB=\RR$,
\cite[Corollary~1]{Benfenati20} follows from items
\cref{t:4ii} and \cref{t:4ivb} in \cref{t:4},
while \cite[Theorem~4.1]{Drus18} follows from items \cref{t:4i} and
\cref{t:4ii} in \cref{t:4} with $\psi=\norm{\Cdot}^2/2$.
\item
Concrete examples of proximity operators of functions on matrix
spaces can be found in
\cite{Livre1,Beck17,Benfenati20,ProxRespo}.
\end{enumerate}
\end{remark}

We now present for the first time an explicit formula for the
proximity operators of Fourier-phase-invariant functions.
Such a formula has thus far been explored only in the case
of projections onto magnitude-based constraint sets in phase
retrieval applications.

\begin{example}[Fourier transform]
\label{ex:86}
Consider the setting of \cref{ex:8}.
Let $\varphi\colon L^2(\RR;\RR)\to\RX$ be proper 
and invariant under pointwise sign-changes in the sense that
\begin{equation}
\brk!{\forall x\in L^2(\RR;\RR)}\quad
\varphi(x)=\varphi\brk{\abs{x}},
\end{equation}
and define
\begin{equation}
\Phi\colon L^2(\RR;\CC)\to\RX\colon f\mapsto
\varphi\circ\abs!{\widehat{f}}.
\end{equation}
Then
\begin{equation}
\brk!{\forall f\in L^2(\RR;\CC)}\quad
\Prox_{\Phi}f
=\menge{\mathfrak{F}^{-1}\brk!{\ee^{\ii\theta}z}}{
z\in\Prox_{\varphi}\abs!{\widehat{f}}
\,\,\text{and}\,\,
\theta\in\AD\,\,\text{with}\,\,
\widehat{f}=\ee^{\ii\theta}\abs!{\widehat{f}}}.
\end{equation}
Particular instances of this formula that manifest themselves in
phase retrieval problems are
\cite[Theorem~4.2]{Luke02} (see also
\cite[Example~3.6]{Bauschke02})
and \cite[Theorem~4\,(1), (6), and (8)]{Youla82}.
\end{example}

As a further illustration \cref{t:4}, we provide an explicit
formula for the proximity operator of block-radial functions
without requiring a separable structure of the underlying function.
This class includes, in particular, the $\ell_{2,p}$ mixed norms,
which are prevalent in multi-task learning and statistics
\cite{Argyriou08,Gramfort12,Yuan06}.

\begin{example}[block-radial function]
\label{ex:8bis-prox}
Consider the setting of \cref{ex:8bis} and, for every $i\in I$,
denote by $\SD_i$ the unit sphere of $\HS_i$.
Let $\varphi\colon\ell^2(I)\to\RX$ be proper 
and invariant under component-wise sign-changes in the sense that
\begin{equation}
\brk!{\forall(\xi_i)_{i\in I}\in\ell^2(I)}\quad
\varphi(\xi_i)_{i\in I}=\varphi\brk!{\abs{\xi_i}}_{i\in I},
\end{equation}
define
\begin{equation}
\Phi\colon\HH\to\RX\colon(\mathsf{u}_i)_{i\in I}\mapsto
\varphi\brk!{\norm{u_i}_{\HS_i}}_{i\in I},
\end{equation}
and let $(\mathsf{u}_i)_{i\in I}\in\HH$.
Then
\begin{equation}
\Prox_{\Phi}(\mathsf{u}_i)_{i\in I}=
\menge{(\zeta_i\mathsf{v}_i)_{i\in I}}{
(\zeta_i)_{i\in I}\in\Prox_{\varphi}
\brk!{\norm{\mathsf{u}_i}_{\HS_i}}_{i\in I}\,\,\text{and}\,\,
(\forall i\in I)\,\,\mathsf{v}_i\in\Proj_{\SD_i}\mathsf{u}_i}.
\end{equation}
\end{example}

By specializing \cref{t:4} to indicator functions,
we obtain explicit expressions for Bregman distances and
Bregman projectors onto spectral sets in terms of those of the
associated invariant sets.

\begin{corollary}
\label{c:5}
Suppose that \cref{a:1} is in force. Let $D$ be a nonempty
$\SD$-invariant subset of $\XX$,
let $\psi\in\Gamma_0(\XX)$ be an $\SD$-invariant Legendre function such
that $D\cap\dom\psi\neq\emp$, and let $X\in\FH$.
Then the following hold:
\begin{enumerate}
\item
\label{c:5i}
$\dist_{\gamma^{-1}(D)}^{\psi\circ\gamma}(X)
=\dist_{D}^{\psi}(\gamma(X))$.
\item
\label{c:5ii}
$\Proj_{\gamma^{-1}(D)}^{\psi\circ\gamma}X
=\menge{\Lambda_{\ad}z}{z\in\Proj_D^{\psi}\gamma(X)\,\,\text{and}\,\,
\ad\in\AD_X}$.
\item
\label{c:5iii}
For every $Z\in\FH$,
\begin{equation}
Z\in\Proj_{\gamma^{-1}(D)}^{\psi\circ\gamma}X
\quad\Leftrightarrow\quad
\begin{cases}
\gamma(Z)\in\Proj_D^{\psi}\gamma(X)\\
(\exi\ad\in\AD)\,\,
X=\Lambda_{\ad}\gamma(X)\,\,\text{and}\,\,
Z=\Lambda_{\ad}\gamma(Z).
\end{cases}
\end{equation}
\item
\label{c:5iv}
Suppose that one of the following is satisfied:
\begin{enumerate}
\item
\label{c:5iva}
$\XX$ is finite-dimensional.
\item
\label{c:5ivb}
$D$ is closed and convex.
\end{enumerate}
Then $\Proj_{\gamma^{-1}(D)}^{\psi\circ\gamma}X$ is a singleton if
and only if $\Proj_D^{\psi}\gamma(X)$ is a singleton.
\end{enumerate}
\end{corollary}
\begin{proof}
Apply \cref{t:4} to the proper $\SD$-invariant function $\iota_D$,
and observe that $\iota_{\gamma^{-1}(D)}=\iota_D\circ\gamma$.
\end{proof}

We end this section with an illustration of \cref{c:5},
which provides a closed-form expression for the projector onto the
intersection of the positive semidefinite cone and a rank
constraint in Euclidean Jordan algebras,
which extends a known formula in the real and complex matrix
settings.

\begin{example}[Euclidean Jordan algebra]
  \label{ex:14}
  Consider the setting of \cref{ex:3}.
  For every $X\in\FH$, we define the \emph{rank of $X$},
  in symbol $\Rank X$, to be the number of nonzero entries of
  the vector $\lambda(X)$ of eigenvalues of $X$.
  Now let $r\in\set{1,\ldots,N-1}$, set
  \begin{equation}
    \DD=\menge{X\in\FH}{\lambda(X)\in\RP^N\,\,\text{and}\,\,
      \Rank X\leq r},
  \end{equation}
  let $X\in\FH$, and recall that (see \cref{e:ax,e:cwpw})
  $\AD_X$ is the set of Jordan frames
  $(A_i)_{1\leq i\leq N}$ in $\FH$ for which
  $X=\sum_{i=1}^N\lambda_i(X)A_i$.
  Furthermore, denote by $\PS_X^N$ the set of $N\times N$ permutation
  matrices that fix $\lambda(X)$, that is,
  \begin{equation}
    \PS_X^N=\menge{P\in\PS^N}{\lambda(X)=P\lambda(X)}.
  \end{equation}
  Then exactly one of the following holds:
  \begin{enumerate}
    \item\label{ex:14i} $(\exi i\in\set{1,\ldots,N})$
      $\lambda_i(X)\geq 0$:
      Set
      \begin{equation}
        q=\max\menge{i\in\set{1,\ldots,N}}{\lambda_i(X)\geq 0}
        \quad\text{and}\quad
        \overline{x}=\brk!{\lambda_1(X),\ldots,\lambda_{\min\set{q,r}}(X),0,\ldots,0}.
      \end{equation}
      Then
      \begin{equation}
        \Proj_{\DD}X
        =\Menge4{\sum_{i=1}^N\zeta_iA_i}{
          (A_i)_{1\leq i\leq N}\in\AD_X,\,\,
          P\in\PS_X^N,\,\,
          (\zeta_i)_{1\leq i\leq N}=P\overline{x}}.
      \end{equation}
    \item\label{ex:14ii} $(\forall i\in\set{1,\ldots,N})$
      $\lambda_i(X)<0$:
      Then $\Proj_{\DD}X=\set{0}$.
  \end{enumerate}
\end{example}
\begin{proof}
  Define
  \begin{equation}
    D=\menge{y\in\RP^N}{\norm{y}_0\leq r},
  \end{equation}
  where $\norm{y}_0$ denotes the number of nonzero entries of a vector
  $y\in\RR^N$.
  Then $D$ is $\PS^N$-invariant and $\DD=\lambda^{-1}(D)$.
  Thus, using \cref{e:cwpw}, we deduce from
  \cref{c:5}\,\cref{c:5ii} applied to the Legendre function
  $\psi=\norm{\Cdot}^2/2$ that
  \begin{equation}
    \Proj_{\DD}X
    =\Menge4{\sum_{i=1}^N\zeta_iA_i}{
      (\zeta_i)_{1\leq i \leq N}\in\Proj_D\lambda(X)\,\,\text{and}\,\,
      (A_i)_{1\leq i\leq N}\in\AD_X},
  \end{equation}
  and it therefore suffices to determine $\Proj_D\lambda(X)$.
  Towards this goal, take $z\in D$ and write (see \cref{n:1})
  \begin{equation}
    \label{e:3gpg}
    z^{\downarrow}=(\pi_i)_{1\leq i\leq N}.
  \end{equation}
  Since $\pi_1\geq\cdots\geq\pi_N\geq 0$ and
  $\norm{z^{\downarrow}}_0=\norm{z}_0\leq r$, we must have
  \begin{equation}
    \label{e:4ahy}
    \pi_1\geq\cdots\geq\pi_r\geq 0=\pi_{r+1}=\cdots=\pi_N.
  \end{equation}

  \cref{ex:14i}:
  Let us consider two cases.

  (a) $q<r$:
  The very definition of $q$ yields
  \begin{equation}
    \label{e:trlh}
    \brk!{\forall i\in\set{q+1,\ldots,r}}\quad\lambda_i(X)<0.
  \end{equation}
  Hence, we derive from the Hardy--Littlewood--P\'olya rearrangement
  inequality that
  \begin{align}
    \norm{z-\lambda(X)}^2
    &\geq\norm!{z^{\downarrow}-\lambda(X)}^2
    \nonumber\\
    &=\sum_{i=1}^q\brk!{\pi_i-\lambda_i(X)}^2
      +\sum_{i=q+1}^r\brk!{\pi_i-\lambda_i(X)}^2
      +\sum_{i=r+1}^N\lambda_i(X)^2
    \nonumber\\
    &\geq\sum_{i=q+1}^r\brk!{\pi_i^2
      -\underbrace{2\pi_i\lambda_i(X)}_{\leq 0}+\lambda_i(X)^2}
      +\sum_{i=r+1}^N\lambda_i(X)^2
    \nonumber\\
    &\geq\sum_{i=q+1}^N\lambda_i(X)^2.
  \end{align}
  In turn, appealing to the condition for equality in the
  rearrangement inequality, we obtain
  \begin{align}
    \norm{z-\lambda(X)}^2=\sum_{i=q+1}^N\lambda_i(X)^2
    &\Leftrightarrow
      \begin{cases}
        \norm{z-\lambda(X)}=\norm!{z^{\downarrow}-\lambda(X)}\\
        \Sum_{i=1}^q\brk!{\pi_i-\lambda_i(X)}^2
        =\Sum_{i=q+1}^r\pi_i^2
        =\Sum_{i=q+1}^r\pi_i\lambda_i(X)=0
      \end{cases}
    \nonumber\\
    &\Leftrightarrow
      \begin{cases}
        \brk!{\exi P\in\PS_X^N}\,\,z=Pz^{\downarrow}\\
        \brk!{\forall i\in\set{1,\ldots,q}}\,\,\pi_i=\lambda_i(X)\\
        \brk!{\forall i\in\set{q+1,\ldots,r}}\,\,\pi_i=0
      \end{cases}
    \nonumber\\
    &\Leftrightarrow
      \brk!{\exi P\in\PS_X^N}\,\,z
      =P\brk!{\lambda_1(X),\ldots,\lambda_q(X),0,\ldots,0}
      =P\overline{x},
  \end{align}
  which entails that
  $\Proj_D\lambda(X)=\menge{P\overline{x}}{P\in\PS_X^N}$.
  
  (b) $q\geq r$: Argue as in case (a).

  \cref{ex:14ii}:
  Upon setting $(\zeta_i)_{1\leq i\leq N}=z$, we deduce that
  \begin{equation}
    \norm{z-\lambda(X)}^2
    =\sum_{i=1}^N\brk!{\zeta_i-\lambda_i(X)}^2
    =\sum_{i=1}^N\brk!{\zeta_i^2
      -\underbrace{2\zeta_i\lambda_i(X)}_{\leq 0}
      +\lambda_i(X)^2}
    \geq\sum_{i=1}^N\lambda_i(X)^2
  \end{equation}
  and, in turn, that
  \begin{equation}
    \norm{z-\lambda(X)}^2=\sum_{i=1}^N\lambda_i(X)^2
    \quad\Leftrightarrow\quad
    \sum_{i=1}^N\zeta_i^2=\sum_{i=1}^N\zeta_i\lambda_i(X)=0
    \quad\Leftrightarrow\quad
    \brk!{\forall i\in\set{1,\ldots,N}}\;\;
    \zeta_i=0.
  \end{equation}
  Thus $\Proj_D\lambda(X)=\set{0}$.
\end{proof}

\begin{remark}
  \label{r:2}
  Here are several comments on \cref{ex:14}.
  \begin{enumerate}
    \item Set
      \begin{equation}
        \mathscr{R}=\menge{X\in\FH}{\Rank X\leq r}
        \quad\text{and}\quad
        \mathscr{C}=\menge{X\in\FH}{\lambda(X)\in\RP^N}.
      \end{equation}
      Along the same lines of the proof of \cref{ex:14},
      one can show that
      \begin{equation}
        \Proj_{\DD}X
        =\Proj_{\mathscr{R}}\brk!{\Proj_{\mathscr{C}}X}.
      \end{equation}
      In general, nevertheless, given nonempty closed subsets $R$ and $C$
      of a real Hilbert space,
      \begin{equation}
        \Proj_{R\cap C}\neq\Proj_R\circ\Proj_C.
      \end{equation}
    \item Choosing $\FH$ to be the Euclidean Jordan algebra of Hermitian
      matrices (see \cref{ex:4}), we obtain an expression for the
      projector onto the set of positive semidefinite matrices of rank
      at most $r$, that is, the set
      \begin{equation}
        \menge{X\in\HS^N(\KB)}{\lambda(X)\in\RP^N\,\,\text{and}\,\,
          \Rank X\leq r},
      \end{equation}
      which is of interest in low-rank factor analysis
      \cite{Bertsimas2017}.
      Such expression is known in the real and complex cases
      (see, e.g., \cite[Lemma~6.1]{Clarkson2017}) but appears to
      be new in the quaternion case.
  \end{enumerate}
\end{remark}

\section{Conclusion}

We have introduced an abstract framework of spectral decomposition
systems that encompasses a wide array of settings
such as classical spectral functions on matrix spaces and Euclidean
Jordan algebras, Fourier-phase-invariant functions,
and block-radial functions.
By explicitly incorporating embedding operators and the spectral
decomposition property, we have established a reduced minimization
principle that enables the constructive evaluation of conjugates,
subgradients, and Bregman proximity operators of
spectral functions in terms of the associated invariant
functions.
These results unify and extend corresponding results in the
individual settings previously studied in isolation, and
the result for set-valued Bregman proximity operators is new, even
in the case of Hermitian matrices.

In the follow-up work \cite{PartII},
we establish corresponding formulae for several generalized
subgradients of nonconvex spectral functions
in the finite-dimensional setting and leverage these results to
extend Lidski\u{\i}'s theorem on the
spectrum of additive perturbations of Hermitian matrices to
finite-dimensional spectral decomposition systems.

Future work will be concerned with applying these representations
for proximal point and splitting algorithms
\cite{Livre1,Beck17,ClasonValkonen} and with extending our results
to objects of second-order variational analysis such as
\cite{Sarabi25}.

\newpage

\bibliographystyle{jnsao}
\bibliography{bibli}

\end{document}